\documentclass[final,1p,times,12pt]{elsarticle}
\makeatletter
\def\ps@pprintTitle{%
   \let\@oddhead\@empty
   \let\@evenhead\@empty
   \let\@oddfoot\@empty
   \let\@evenfoot\@oddfoot
}
\makeatother
\usepackage{fullpage}
\usepackage[colorlinks]{hyperref}
\usepackage[utf8]{inputenc}
\usepackage{url}
\usepackage{hyperref}

\biboptions{square,sort,comma,numbers}
\usepackage{graphicx}
\usepackage{amscd,amsmath,amsthm,amssymb}
\usepackage{verbatim}
\usepackage[all]{xy}
\usepackage{color}
\usepackage[colorlinks]{hyperref}

\usepackage{tikz, float} \usetikzlibrary {positioning}
\usetikzlibrary{calc}

\usepackage{makeidx}         
\usepackage{graphicx}        
\usepackage{multicol}        
\usepackage[bottom]{footmisc}

\usepackage[ruled, linesnumbered]{algorithm2e} 

\usepackage{array}

\usepackage{newtxmath}    
\makeindex

\newcommand{\BLambda}{\mathcal{B}_\ell}
\newcommand{\id}{\text{id}}
\newcommand{\ini}{\text{in}}

\def\frk{\mathfrak}               

\def\Phi{{\frk n}}
\def\Phi{{\frk N}}

\def\wb{{\mathbf w}}

\def\B{{\mathcal B}}

\def\A{{\mathcal A}}

\def\Tc{{\mathcal T}}

\def\opn#1#2{\def#1{\operatorname{#2}}} 
\opn\chara{char} \opn\length{\ell} \opn\pd{pd} \opn\rk{rk}
\opn\projdim{proj\,dim} \opn\injdim{inj\,dim} \opn\rank{rank}
\opn\depth{depth} \opn\grade{grade} \opn\height{height}
\opn\embdim{emb\,dim} \opn\codim{codim}

\opn\Tr{Tr} \opn\bigrank{big\,rank}
\opn\superheight{superheight}\opn\lcm{lcm}
\opn\trdeg{tr\,deg}
\opn\reg{reg} \opn\lreg{lreg} \opn\ini{in} \opn\lpd{lpd}
\opn\size{size} \opn\sdepth{sdepth}
\opn\link{link}\opn\fdepth{fdepth}\opn\lex{lex}
\opn\LM{LM}
\opn\LC{LC}
\opn\NF{NF}
\opn\Merge{Merge}
\opn\sgn{sgn}

\opn\div{div} \opn\Div{Div} \opn\cl{cl} \opn\Pic{Pic}
\opn\Prin{Prin}
\opn\op{op}
\opn\indeg{indeg} \opn\outdeg{outdeg}
\opn\red{red}
\opn\Spec{Spec} \opn\Supp{Supp} \opn\supp{supp} \opn\Sing{Sing}
\opn\Ass{Ass} \opn\Min{Min}\opn\Mon{Mon} \opn\val{val}
\opn\Ann{Ann} \opn\Rad{Rad} \opn\Soc{Soc}
 \opn\Ker{Ker} \opn\Coker{Coker} \opn\Am{Am}
\opn\Hom{Hom} \opn\Tor{Tor} \opn\Ext{Ext} \opn\End{End}
\opn\Aut{Aut} \opn\id{id}

\opn\nat{nat}
\opn\pff{pf}
\opn\Pf{Pf} \opn\GL{GL} \opn\SL{SL} \opn\mod{mod} \opn\ord{ord}
\opn\Gin{Gin} \opn\Hilb{Hilb}\opn\sort{sort}
\opn\Image{Image}
\opn\vol{Vol}
\opn\aff{aff} \opn\con{conv} \opn\relint{relint} \opn\st{st}
\opn\lk{lk} \opn\cn{cn} \opn\core{core} \opn\vol{vol}
\opn\link{link} \opn\star{star}\opn\lex{lex}\opn\set{set}
\opn\dist{dist}
\opn\gr{gr}

\def\pot#1#2{#1[\kern-0.28ex[#2]\kern-0.28ex]}

\opn\dirlim{\underrightarrow{\lim}}
\opn\inivlim{\underleftarrow{\lim}}
\def\Implies{\ifmmode\Longrightarrow \else
        \unskip${}\Longrightarrow{}$\ignorespaces\fi}
\def\implies{\ifmmode\Rightarrow \else
        \unskip${}\Rightarrow{}$\ignorespaces\fi}
\def\iff{\ifmmode\Longleftrightarrow \else
        \unskip${}\Longleftrightarrow{}$\ignorespaces\fi}

\let\:=\colon
\newtheorem{theorem}{Theorem}[section]
\newtheorem{lemma}[theorem]{Lemma}
\newtheorem{corollary}[theorem]{Corollary}
\newtheorem{proposition}[theorem]{Proposition}

\newtheorem{question}[theorem]{Question}
\theoremstyle{remark}
\newtheorem{remark}[theorem]{Remark}

\theoremstyle{definition}
\newtheorem{example}[theorem]{Example}

\newtheorem{definition}[theorem]{Definition}

\DeclareMathOperator{\Gr}{Gr}
\DeclareMathOperator{\Flag}{Fl}

\let\kappa=\varkappa

\def\qed{\ifhmode\textqed\fi
      \ifmmode\ifinner\quad\qedsymbol\else\dispqed\fi\fi}
\def\textqed{\unskip\nobreak\penalty50
       \hskip2em\hbox{}\nobreak\hfil\qedsymbol
       \parfillskip=0pt \finalhyphendemerits=0}
\def\dispqed{\rlap{\qquad\qedsymbol}}

\opn\dis{dis}
\def\pnt{{\raise0.5mm\hbox{\large\bf.}}}

\opn\Lex{Lex}
\opn\syz{{\rm syz}}
\opn\spoly{{\rm spoly}}
\opn\LM{{\rm LM}}
\opn\lm{{\rm lm}}
\opn\lcm{{\rm lcm}} \opn\A{\mathcal A}

\numberwithin{equation}{section}

\newcommand{\inwb}{{\rm in}_{{\bf w}_\ell}}

\newcommand{\II}{\mathcal{I}}
\DeclareMathOperator{\trop}{trop}

\DeclareMathOperator{\init}{in}

\usepackage{multirow}

\begin{document}

\begin{frontmatter}

\title{Toric degenerations of Grassmannians and Schubert varieties \\  from matching field tableaux\\ }
\author{
Oliver Clarke and Fatemeh Mohammadi
}
\begin{abstract}
We study Gr\"obner degenerations of Grassmannians and the Schubert varieties inside them. We provide a family of binomial ideals whose combinatorics is governed by matching field tableaux in the sense of Sturmfels and Zelevinsky in \cite{sturmfels1993maximal}.
We prove that these ideals are all quadratically generated and they yield a SAGBI basis of the Pl\"ucker algebra. This leads to a new family of toric degenerations of Grassmannians. Moreover, we apply our results to construct a family of Gr\"obner degenerations of Schubert varieties inside Grassmannians. We provide a complete characterization of toric ideals among these degenerations in terms of the combinatorics of matching fields, permutations and semi-standard tableaux. \end{abstract}

\end{frontmatter}

\section{Introduction}\label{sec:1}
Computing toric degenerations of varieties is
a powerful tool to extend the deep relationship between combinatorics and toric varieties to more general
spaces \cite{An13}. Given a projective variety, a toric degeneration is a flat family whose fiber over zero is a toric variety
and all of whose other fibers are isomorphic to the original variety. Hence, toric degenerations enable us to use
the tools developed in toric geometry to study more general varieties. 

\medskip

Toric degenerations of Grassmannians, flag and Schubert varieties have been extensively studied in the literature, see e.g. \cite{gonciulea1996degenerations,caldero2002toric,fang2017toric}. 
The main example of such degenerations is  the Gelfand-Cetlin degeneration which is studied with semi-standard tableaux and Gelfand-Cetlin polytopes \cite{FvectorGC, KOGAN}. For the Grassmannian, this example has been generalized, see e.g.~ \cite{Speyer, Witaszek, rietsch2017newton, BFFHL}. 
For example, Rietsch and  Williams describe a family of toric degenerations of Grassmannians arising from plabic graphs \cite{rietsch2017newton}. 
Recently, Kaveh and Manon have used tools from tropical geometry to study toric degenerations of ideals in general \cite{KM16}. The main idea is that the initial ideals associated to the top-dimensional facets of tropicalizations of ideals are good candidates for toric degenerations. A similar approach has been taken in studying toric degenerations of $\Gr(3,n)$ in \cite{KristinFatemeh}, and small flag varieties in \cite{bossinger2017computing}. More precisely, tropical Grassmannians, defined by Speyer and Sturmfels in \cite{Speyer}, provide a nice framework for studying toric degnerations of Grassmannians. 
In \cite{KristinFatemeh}, Mohammadi and Shaw have used this framework together with the theory of matching fields from \cite{sturmfels1993maximal} to show that every coherent matching field has a canonical toric ideal that can be identified as the toric component of the initial ideal associated to a top-dimensional facet of $\trop(\Gr(k,n))$. This, in particular, leads to a family of toric degenerations for $\Gr(3,n)$. Our work extends results from \cite{KristinFatemeh} to higher-dimensional Grassmannians.
\medskip

In this work, we are interested in toric degenerations of Grassmannians and Schubert varieties inside them from the point-of-view of algebraic combinatroics. In other words, we have a positive answer to \textcolor{blue}{the} {\it Degeneration Problem}, posed by Caldero in \cite{caldero2002toric}, for Schubert varieties inside  Grassmannians (see \cite{fang2017toric}, and references therein, for other such examples). 

\medskip

Let $\Lambda(k,n)$, or $\Lambda$ when there is no confusion, be a set of $k\times 1$ tableaux of integers corresponding to all  $k$-subsets of $[n]=\{1,\ldots,n\}$.  By combining 
tableaux from $\Lambda$ side by side, we can construct larger tableaux.  We say a pair of row-wise equal tableaux differ by a {\em swap} if they are the same for all but two columns. Moreover, two row-wise equal tableaux 
are called {\it quadratically equivalent} 
if they can be obtained from each other by a series of swaps. 
Now, let $R= \mathbb{K}[{\bf P}]$ denote the polynomial ring with one variable associated to 
each $k$-subset of $[n]$. For any pair of quadratically equivalent tableaux $T$ and $T'$ one can write a binomial ${\bf P}^{T}-{\bf P}^{T'}.$ The ideal $J_{\Lambda}\subset R$
is generated by all such binomials. In other words, two tableaux are quadratically equivalent if and only if their corresponding monomials are equal in $R/J_\Lambda$. This ideal is implicitly defined by Sturmfels and Zelevinsky to address questions about determinantal varieties; $\Lambda$ is the matching field defined in \cite{sturmfels1993maximal}. In \cite{fink2015stiefel}, Fink and Rinc\'on provided a link between tropical linear spaces and a specific family \textcolor{blue}{of} so-called {\em coherent} matching fields. From a combinatorial point-of-view, it is much more convenient to work with the tableau description of a matching field than the classical definition. In \cite{KristinFatemeh} Mohammadi and Shaw used the tableaux description of coherent matching fields to study the tropicalization of Grassmannians and provided a necessary combinatorial condition for a top-dimensional facet of $\trop(\Gr(k,n))$ to lead to a toric degeneration of $\Gr(k,n)$. When $J_\Lambda$ is quadratically generated, this condition is also sufficient. In particular, for $\Gr(3,n)$, the authors of \cite{KristinFatemeh} provided a family of matching fields, called $2$-block diagonal, whose ideals are quadratically generated and hence lead to a family of toric degenerations of $\Gr(3,n)$. Our main result generalizes this to higher-dimensional Grassmannians. Our natural generalization of block diagonal matching fields as slight perturbations of the classical Gelfand-Cetlin degeneration lead to a family of toric degenerations for $\Gr(k,n)$ for $k\geq 3$. A priori it is not clear why such a family of matching fields should produce toric degenerations or why the resulting varieties are different than the classical case. However, we show that, quite surprisingly, this construction leads to toric varieties many of which are non-isomorphic.

\medskip

In this paper, we study the Pl\"ucker ideals of $\Gr(k,n)$,  denoted by $G_{k,n}$, and their associated algebras from the point-of-view of Gr\"obner basis theory and SAGBI basis theory. Namely, we extend the family of block diagonal matching fields defined in \cite{KristinFatemeh} from  $\Gr(3,n)$ 
to higher-dimensional Grassmannians (see Definition~\ref{def:block}). 
Then, we explicitly construct a weight vector for every such matching field and we study its corresponding initial ideal
$\ini_\wb(G_{k,n})$ (see Definition~\ref{def:initial}). We prove that such ideals are equal to their corresponding matching field ideals $J_\Lambda$ and, in particular, they are all toric.  We remark that a general matching field rarely gives rise to a toric or even a binomial initial ideal. 
Moreover, we prove that the initial ideals
$\ini_\wb(G_{k,n})$ are all quadratically generated.
Note that describing a minimal generating set of toric ideals or even finding an upper bound for the degree of the generators is a difficult problem. Such questions are usually studied for  special families of ideals with combinatorial structures, see e.g.~\cite{White, hibi1987distributive,ene2011monomial,Ohsugi, Mateusz, Cone}. 

\medskip

We describe a minimal generating set of the associated Pl\"ucker algebra of $\ini_\wb(G_{k,n})$ in terms of its corresponding {\em matching field tableaux}. 
We find this generating set by explicitly constructing a SAGBI basis for the Pl\"ucker algebra. 
In combinatorial terms, for each matching field we construct a collection of $k\times 2$ tableaux such that every $k\times 2$ tableau is row-wise equal to exactly one tableau in the collection. Then, we show that this collection is in bijection with $k\times 2$ tableaux in semi-standard form, which provide a SAGBI basis for the Pl\"ucker algebra in the Gelfand-Cetlin case. Hence, our results directly generalize the analogous results from the classical Gelfand-Cetlin case, see e.g.~\cite[Theorem 14.11]{MS05}.

\medskip

The paper is structured as follows: 
In \S\ref{prelim}, we fix our notation and introduce our main objects of study. We define the Pl\"ucker ideal, block diagonal matching fields and their associated ideals. We also define Schubert varieties and their Gr\"obner degenerations using matching field ideals (see Definition~\ref{def:ideals}).
In \S\ref{sec:tableaux}, we define tableaux arising from matching fields. Basic properties of tableaux are studied in \S\ref{sec:tableaux_quad}. Moreover, in \S\ref{sec:core}, we study the matching field tableaux from the point-of-view of algebras, and in Lemmas~\ref{lem:basis_span} and \ref{lem:basis_indep}, we provide a SAGBI basis for the Pl\"ucker algebras associated to block diagonal matching fields. These are the core lemmas needed for our main results. 
In \S\ref{sec:Grassmannian}, we study the ideals of block diagonal matching fields from the viewpoint of SAGBI basis theory and Gr\"obner basis theory. In Theorem~\ref{prop:quad}, we show that they are all quadratically generated.
Then, we study their associated initial algebras and in Theorem~\ref{thm:SAGBI}, we prove that the Pl\"ucker variables form a SAGBI basis for the Pl\"ucker algebra. Hence, we obtain a family of toric degenerations of $\Gr(k,n)$.  We then turn our attention to Schubert varieties in \S\ref{sec:Schubert}. Our goal in studying degenerations of Schubert varieties is to answer Question~\ref{question:grassmannian}. Our main result is Theorem~\ref{Intro:Grassmannian} in which we provide a complete characterization of toric ideals arising from block diagonal matching fields for Schubert ideals.

\smallskip

\noindent{\bf Acknowledgement.} The second author would like to express her gratitude to Kristin Shaw and Bernd Sturmfels for introducing her to this concept and for many helpful conversations.
We would also like to thank Narasimha Chary and J\"urgen Herzog
for helpful comments on a preliminary version of this manuscript, and Elisa Gorla for fruitful conversation on the proof of Theorem~\ref{thm:SAGBI}. 
The first author was supported by EPSRC Doctoral Training Partnership 
(DTP) award EP/N509619/1.
The second author was supported by a BOF Starting Grant of Ghent University and EPSRC Early Career Fellowship EP/R023379/1. We are grateful to the anonymous referees for very helpful comments on earlier versions of this paper.


\section{Preliminaries}\label{prelim} 
Throughout we set  $[n]:= \{1, \dots , n\}$ and we use  $\mathbf{I}_{k, n}$ to denote the collection of subsets of $[n]$ of size $k$. The symmetric group on $n$ elements is denoted by $S_n$. We also fix a field $\mathbb{K}$ with char$(\mathbb{K})=0$. We are mainly interested in the case when
$\mathbb{K}=\mathbb{C}$.

\subsection{\bf Flag varieties and Grassmannians}\label{sec:prelim_flag_grassmannian}  
The set of full flags 
$$\{0\}= V_0\subset V_1\subset\cdots\subset V_{n-1}\subset V_n=\mathbb{K}^n$$
of vector subspaces of $\mathbb{K}^n$ with ${\rm dim}_{\mathbb{K}}(V_i) = i$ is called the {\em flag variety} denoted by $\Flag_n$, which is naturally embedded in a product of Grassmannians using Pl\"ucker variables.
Each point in the flag variety can be represented by an $n\times n$ matrix $X=(x_{i,j})$ whose first $k$ rows generate $V_k$ which corresponds to a point in the Grassmannian $\Gr(k,n)$. Therefore, the ideal of $\Flag_n$, denoted by $I_n$, is the kernel of the map
\[\label{eqn:pluckermap} 
\varphi_n:\  \mathbb{K}[P_J:\ \varnothing\neq J\subsetneq \{1,\ldots,n\}]\rightarrow \mathbb{K}[x_{i,j}:\ 1\leq i\leq n-1,\ 1\leq j\leq n]
\]
sending each Pl\"ucker variable
$P_J$ to the determinant of the submatrix of $X$ with row indices $1,\ldots,|J|$ and column indices in $J$. 
Similarly, we define the {\em Pl\"ucker ideal} of $\Gr(k,n)$, denoted by $G_{k,n}$, as the kernel of the map $\varphi_n$ restricted to the ring with variables $P_J$ with $|J|=k$.

\subsection{\bf Schubert varieties}\label{sec:prelim_schubert}
Let SL$(n,\mathbb C)$ be the set of $n\times n$ matrices with determinant $1$,  and let $B$ be its subgroup consisting of upper triangular matrices.  There is a natural transitive action of SL$(n,\mathbb C)$ on the flag variety $\Flag_n$ which identifies $\Flag_n$ with the set of left cosets SL$(n,\mathbb C)/B$, since $B$ is the stabilizer of the standard flag
$0\subset \langle e_1\rangle \subset\cdots \subset \langle e_1, \ldots, e_n\rangle=\mathbb C^n $. Given a permutation $w\in S_n$, $\sigma_w$ is the $n\times n$ matrix whose only non-zero entries are $1$ in position $(w(i),i)$ for each $1 \le i \le n$.
 By the Bruhat decomposition, we can write 
${\rm SL}(n,\mathbb C)/B= \coprod_{w\in S_n}B\sigma_wB/B.$

Given a permutation $w$, its {\em Schubert variety} is $$X(w)=\overline{B\sigma_wB/B} \subseteq \Flag_n$$ which is the Zariski closure of the corresponding cell in the Bruhat decomposition. 
The associated ideal of the Schubert variety $X(w)$ is
$I(X(w))=I_n+\langle P_I: I\in S_w\rangle,$
where 
$$
S_w=\{I : I\subset[n] \ \text{with} \ I=(i_1,\ldots,i_{|I|})\not\leq (w_{\ell_1},w_{\ell_2},\ldots,w_{\ell_{|I|}})\},
$$
and $w_{\ell_1}<w_{\ell_2}<\cdots<w_{\ell_{|I|}}$ is obtained by ordering the first $\lvert I \rvert$ entries of $w$. Here, $I_n$ is the Pl\"ucker ideal defined in \S 2.1 and $\le$ is the component-wise partial order on $\mathbf{I}_{k, n}$. 

Similarly, we can study the Schubert varieties inside $\Gr(k,n)$. The permutations giving rise to distinct Schubert varieties in $\Gr(k,n)$ are of the form $w=(w_1,\ldots,w_n)$ 
where  $w_{1}<\cdots<w_k$, $w_{k+1}<\cdots<w_n$. Therefore, it is enough to record the permutations of $S_n$ as $w=(w_1,\ldots,w_k)$ which will be called a {\em Grassmannian permutation}. If $w$ is a Grassmannian permutation then we define $S_{w,k}=S_w\cap\{I:\ |I|=k\}$. The elements in $S_{w,k}$ correspond to the variables which vanish in the ideal of the Schubert variety of the Grassmannian, see Definition~\ref{def:ideals}.

\subsection{\bf Matching fields}\label{sec:prelim_matching_field}
Some of the most important features of this section are that each matching field induces a canonical toric ideal, and gives rise to a weight vector for the Pl\"ucker variables $P_I$ for $I\in \mathbf{I}_{k, n}$. Following \cite{KristinFatemeh}, we define matching fields as follows.

Given integers $k$ and $n$, a matching field denoted by $\Lambda(k,n)$, or $\Lambda$ when there is no confusion, is a choice of permutation $\Lambda(I) \in S_k$ for each $I \in \mathbf{I}_{k, n}$. We think of the permutation {$\sigma=\Lambda(I)$} as inducing a new ordering on the elements of $I$, 
where the position of $i_s$ is  $\sigma(s)$. In addition, we think of $I$ as being identified with a monomial of the Pl\"ucker form $P_I$ and we represent this monomials as a $k \times 1$ tableau where the entry of $(\sigma(r), 1)$ is $i_{r}$. To make this tableau notation precise we define the ideal of the matching field as follows.

Let $X=(x_{i,j})$ be a $k \times n$ matrix of indeterminates. To every $k$-subset $I$ of $[n]$ with $\sigma=\Lambda(I)$ we associate the monomial 
$
\textbf{x}_{\Lambda(I)}:=x_{\sigma(1) i_{1}}x_{\sigma(2)i_2}\cdots x_{{\sigma(k)i_k}}. 
$
The {\em matching field ideal} $J_\Lambda$ is defined as the kernel of the monomial map
\begin{eqnarray}\label{eqn:monomialmap}
\phi_{\Lambda} \colon\  & \mathbb{K}[P_I]  \rightarrow \mathbb{K}[x_{ij}]  
\quad\text{with}\quad
 P_{I}   \mapsto \text{sgn}(\Lambda(I)) \textbf{x}_{\Lambda(I)},
\end{eqnarray}
where $\text{sgn}(\Lambda(I))$ denotes the signature of the permutation $\Lambda(I)$ for each $I \in \mathbf{I}_{k, n}$. 
\begin{definition}
A matching field $\Lambda$  is \emph{coherent} if there exists an $k\times n$ matrix $M$ with entries in $\mathbb{R}$ 
such that 
for every  $I \in \mathbf{I}_{k,n}$ 
the initial of the Pl\"ucker form  $P_I \in \mathbb{K}[x_{ij}]$ is $\text{in}_M (P_I) = \phi_{\Lambda}(P_I)$, where $\text{in}_M (P_I)$ is the sum of all terms in $P_I$ of the lowest weight and the weight of a monomial $\bf m$ is the sum of entries in $M$ corresponding to the variables in $\bf m$.
In this case, we say that the matrix $M$ \emph{induces the matching field} $\Lambda$. The weight of each variable $P_I$ is defined as the minimum weight of the terms of the corresponding minor of $M$, and it is called {\em the weight induced by $M$}.
\end{definition}
\begin{example}\label{ex:diag}
Consider the matching field $\Lambda(3,5)$ which assigns to each subset $I$ the identity permutation. Consider the following matrix:
\[
M=\begin{bmatrix}
     0  & 0  & 0  & 0  & 0  \\
     5  & 4 & 3  & 2  & 1\\
     9 & 7 & 5 & 3  & 1 \\
\end{bmatrix}.
\]
The weights induced by $M$ on the variables $P_{123}, P_{124},\ldots,P_{345}$ are $9, 7, 5, 6, 4, 3, 6, 4,3, 3$, respectively. Thus, for each $I=\{i,j,k\}$ we have that $\text{in}_M (P_I) = x_{1i}x_{2j}x_{3k}$ for $1\leq i<j<k\leq 5$. Therefore, the matrix $M$ induces $\Lambda(3,5)$. 
Below are the tableaux representing $P_I$ for each $I$:
$$\begin{array}{c}1 \\2  \\ 3 \end{array} , \quad
\begin{array}{c} 1 \\2  \\4 \end{array},\quad
\begin{array}{c}1  \\2 \\ 5 \end{array} , \quad
\begin{array}{c}1 \\ 3 \\ 4  \end{array} ,\quad
\begin{array}{c}1 \\ 3  \\ 5\end{array} ,\quad
\begin{array}{c}1 \\ 4 \\ 5\end{array} ,\quad
\begin{array}{c}2 \\ 3\\ 4  \end{array} ,\quad
\begin{array}{c}2 \\ 3\\ 5\end{array} ,\quad
\begin{array}{c} 2 \\ 4   \\ 5 \end{array} ,\quad
\begin{array}{c} 3 \\ 4   \\ 5 \end{array}. 
$$
\end{example}
Notice that each initial term $\text{in}_M(P_I)$ arises from the leading diagonal. Such matching fields are called {\em diagonal}. See, e.g. \cite[Example 1.3]{sturmfels1993maximal}. 
\begin{definition}\label{def:block}
Given $k,n$ and $0\leq\ell\leq n$, we define the 
{\em block diagonal matching field}
$\BLambda$
as the map from $\mathbf{I}_{k,n}$ to $S_k$ such that
\[
 \BLambda(I)= \left\{
     \begin{array}{@{}l@{\thinspace}l}
      id  &: \text{if $\lvert I|=1$ or $\lvert I \cap \{1,\ldots,\ell\}\rvert \neq 1$},\\
      (12)  &: \text{otherwise}. \\
     \end{array}
   \right.
\]
These matching fields are called $2$-block diagonal in \cite{KristinFatemeh}. Note that $\ell=0$ or $n$ gives rise to the classical \emph{diagonal matching field} as in Example~\ref{ex:diag}. Given a block diagonal matching field $\BLambda$ we define $B_{\ell,1} = \{1, \dots, \ell \}$ and $B_{\ell,2} = \{\ell + 1, \dots, n \}$.
\end{definition}

\begin{example}\label{example:block}
Given $k,n$ and $0\leq \ell\leq n$, the matching field $\BLambda$ is induced by the matrix:
\[
M_{\ell}=\begin{bmatrix}
    0       & 0         & \cdots    & 0         &  0      & 0      & \cdots     & 0  \\
    \ell    & \ell-1    & \cdots    & 1         &n        &n-1     &\cdots      &\ell+1\\
    2n      & 2(n-1)    & \cdots    & 10        &  8      & 6      &4           & 2  \\
    \vdots  & \vdots    & \ddots    & \vdots    & \vdots  & \vdots &  \vdots    &  \vdots   \\
     n(k-1) & (n-1)(k-1)& \cdots    & 5(k-1)    & 4(k-1)  & 3(k-1) & 2(k-1)     & k-1    \\
\end{bmatrix},
\]
and hence, it is coherent. We denote $\wb_\ell$ for the weight vector induced by $M_\ell$.
\end{example}

\begin{definition}\label{def:initial}
Given a block diagonal matching field $\BLambda$, we denote 
the initial ideal 
of $G_{k,n}$ with respect to $\wb_\ell$ by $\inwb(G_{k,n})$ and we define it as the ideal generated by polynomials $\inwb(f)$ for all $f\in G_{k,n}$, where 
\[\inwb(f)=\sum_{\alpha_j\cdot \wb_\ell=d}{c_{{\bf \alpha}_j}\bf P}^{{\bf \alpha}_j}\quad\text{for}\quad f=\sum_{i=1}^t c_{{\bf \alpha}_i}{\bf P}^{{\bf \alpha}_i}\quad\text{and}\quad d=\min\{\alpha_i\cdot \wb_\ell:\ i=1,\ldots,t\}.\]
\end{definition}

\begin{definition}\label{def:ideals}
Given a block diagonal matching field $\BLambda$ and a permutation $w\in S_n$ we define the ideal $G_{k,n,\ell,w}$ to be the ideal obtained by setting the variables $\{P_I : I \in S_{w, k} \}$ to be zero in ${\rm in}_{{\bf w}_\ell}(G_{k,n})$.

This can be computed in $\mathtt{Macaulay2}$ \cite{M2} by using the command $\mathtt{substitute}$ or equivalently $\mathtt{sub}$:
\[
G_{k,n,\ell,w} = \mathtt{sub}({\rm in}_{{\bf w}_\ell}(G_{k,n}), \{P_I => 0 : I \in S_{w,k}\} ).
\]
This algorithm works by computing a Gr\"obner basis and performing elimination. So the ideal $G_{k,n,\ell,w}$ can be obtained by adding the variables $P_I$, where $I\in S_{w,k}$ and then eliminating them which can be computed
\[
G_{k,n,\ell,w}={\rm eliminate\ }({\rm in}_{{\bf w}_\ell}(G_{k,n})+\langle P_I:\ I\in S_{w,k}\rangle,\ \{P_I:\ I\in S_{w,k}\}).
\]
We say that the variable $P_I$ vanishes in $G_{k,n,\ell,w}$ if $I \in S_{w,k}$. 
\end{definition}

\section{Matching field tableaux}\label{sec:tableaux} 
The tableaux which have arisen in the theory of matching fields in \cite{sturmfels1993maximal,KristinFatemeh} will be the main tool in the proofs of our main results in \S\ref{sec:Grassmannian}. Here, we prove important properties about these tableaux.
This section, while elementary, is the most technical part of the paper. We provide illustrative examples and diagrams to make it easier to follow the proofs. The main ingredients needed for our main theorems in \S\ref{sec:Grassmannian}, about Grassmannians, are Lemmas~\ref{lem:swaps_case1} and \ref{lem:swaps_case2} for Theorem~\ref{prop:quad}, and Lemmas~\ref{lem:basis_span}, \ref{lem:basis_indep} and \ref{lem:basis_bijection} for Theorem~\ref{thm:SAGBI}. Other results are used to establish these main ingredients.

\medskip

Throughout this section we fix a block diagonal matching field $\BLambda$. For each collection $\II = \{I_1, \dots, I_t\}$ of non-empty subsets of $[n]$ we denote by $T_{\II}$ or, when there are few columns, by $T_{I_1 \dots I_t}$ \label{notation:tableau} the tableau with columns $I_i$. The order of the elements in each column is given by the matching field $\BLambda$.

\subsection{\bf Quadratic relations among tableaux}\label{sec:tableaux_quad}

\smallskip

\begin{definition}\label{def:column_type}
Given a tableau $T$, for each column $I \in T$ we say that $I$ is of \emph{type} $|I \cap B_{\ell,1}|$. 
\end{definition}
\begin{figure}
    \centering
    \begin{tabular}{|cccccc|cc|}
        \multicolumn{6}{c}{$T_X$} & \multicolumn{2}{c}{$T_Y$} \\
        \hline
        1       & 1     & 2     & 2     & 3     & 5     & 4     & 5  \\
        2       & 2     & 3     & 4     & 4     & 6     & 1     & 3  \\
        3       & 4     & 4     & 5     & 6     & 7     & 5     & 6  \\
        4       & 6     & 7     & 7     & 8     & 8     & 7     & 8  \\
        \hline
    \end{tabular}
    \caption{An example of a tableau for $\Gr(4,8)$ with block diagonal matching field $B_4=(1234|5678)$. The columns appearing in the table from left to right are of type $4,3,3,2,2,0,1,1$. The tableau is partitioned into $T_X$, whose column entries appear in increasing order, on the left and $T_Y$ on the right. 
    }
    \label{tab:type_eg}
\end{figure}
Importantly, if two tableaux are row-wise equal, then they have the same number of columns of each type. Also the columns in $T$ whose entries do not appear in increasing order are exactly those of type $1$. We distinguish these columns with the following notation.

Let $X$ be the collection of all $I \in \mathbf{I}_{k, n}$ for which $\BLambda(I) = id$ and let $Y$ be the collection of all $I \in \mathbf{I}_{k, n}$ for which $\BLambda(I) = (12)$. Equivalently, $Y$ is the collection of all $I$ of type $1$ and $X$ is the collection of all $I$ of type $0, 2,3, \dots, k$. We denote by $T_X$ the sub-tableau of $T$ whose columns lie in $X$ and similarly $T_Y$ the columns of $T$ which lie in $Y$. By convention, we write $T = [T_X \mid T_Y]$ with $T_X$ on the left and $T_Y$ on the right. For example, see Figure~\ref{tab:type_eg}.

\begin{definition}\label{def:swap}
Suppose that $T$ and $T'$ are two tableaux that are row-wise equal.
We say that $T$ and $T'$ differ by a {\em swap}, or $T'$ is obtained from $T$ by a {\em quadratic relation}, if $T$ and $T'$ are the same for all but two columns. If $T$ and $T'$ differ by a sequence of swaps, then we say $T$ and $T'$ are {\em quadratically equivalent}.
\end{definition}

\begin{example} \label{example:semi_std_basis}
Consider the diagonal matching field for $\Gr(3,6)$
and the tableaux $T_1$ and $T_3$ below. Then $T_1$ is quadratically equivalent to $T_3$ with the following sequence of tableaux:
\[
T_1 =
\begin{tabular}{|ccc|}
    \hline
    2 & 1 & 3 \\
    3 & 2 & 4 \\
    4 & {\bf\textcolor{blue}{6}} & {\bf\textcolor{blue}{5}} \\
    \hline
\end{tabular} \, ,
\quad
T_2 = 
\begin{tabular}{|ccc|}
    \hline
    {\bf\textcolor{blue}{2}} & {\bf\textcolor{blue}{1}} & 3 \\
    {\bf\textcolor{blue}{3}} & 
    {\bf\textcolor{blue}{2}} & 4 \\
    4 & 5 & 6 \\
    \hline
\end{tabular} \, ,
\quad
T_3 =
\begin{tabular}{|ccc|}
    \hline
    1 & 2 & 3 \\
    2 & 3 & 4 \\
    4 & 5 & 6 \\
    \hline
\end{tabular}\ .
\]
We say that a tableau is in \emph{semi-standard} form if its columns are strictly increasing and its rows are weakly increasing. In this example, $T_3$ is in semi-standard form. More generally for the diagonal matching field, if $T$ is a tableau then there exists a unique semi-standard tableau $T'$ which is row-wise equal to $T$. Moreover, $T'$ is obtained from $T$ by a sequence of quadratic relations. This follows immediately by applying the following construction to every pair of columns in $T$. For any two  $k$-subsets $I = \{i_1 < i_2 < \dots < i_k\}$ and $J = \{j_1 < j_2 < \dots < j_k \}$, we let $I' = \{\min\{i_1, j_1\}, \min\{i_2, j_2\}, \dots, \min\{i_k, j_k\} \}$ and $J' = \{\max\{i_1, j_1\}, \max\{i_2, j_2\}, \dots, \max\{i_k, j_k\} \}$. Then $T_{IJ}$ and $T_{I'J'}$ are row-wise equal, and $T_{I'J'}$ is in semi-standard form. 
\end{example}

Let $T$ be a tableau. By the above, we may change $T_X$, by a series of quadratic relations, into semi-standard form. Similarly we may transform $T_Y$ so that the entries along each row are weakly increasing and we call this {\em the semi-standard form of} $T_Y$. Unless stated otherwise, we assume that any tableau $T$ is written with each part, $T_X$ and $T_Y$, in semi-standard form, see, e.g. Figure~\ref{tab:type_eg}.

In the following two lemmas, we show how to make row-wise equal tableaux 
similar, by a sequence of quadratic relations.

\begin{lemma}\label{lem:swaps_case1}
Suppose that $T$ and $T'$ are row-wise equal tableaux and the leftmost column of each is of type $i$ where $i \in \{2, 3, \dots, k \}$. Then there exist tableaux $S$ and $S'$ quadratically equivalent to $T$ and $T'$, respectively, such that $S$ and $S'$ contain an identical column of type $i$.
\end{lemma}

\begin{proof}
Let $A = (a_1, \dots, a_k)^{tr}$ be the leftmost column of $T$ and assume that it does not appear as a column in $T'$. Let $B$ be the leftmost column of $T'$ and similarly assume that it does not appear in $T$. Note that $A$ and $B$ have the same type and
 $$B = (a_1, b_2, a_3, \dots, a_i, b_{i+1}, \dots, b_{k})^{tr} \text{ for some } b_{i+1}, \dots, b_k.$$
If $b_2 \neq a_2$ then, without loss of generality, suppose that $a_2 < b_2$. So $a_2$ appears in the second row of $T'_Y$ and we may swap $a_2$ and $b_2$ in $T'$. Thus, we may assume that $b_2 = a_2$.

Suppose $j$ is the smallest index for which $b_j \neq a_j$. Without loss of generality we may assume that $a_j < b_j$. So $a_j$ must appear in the $j^{\rm th}$ row of $T'_Y$. Suppose $a_j$ appears in column $C = (c_1, \dots, c_k) \in T'_Y$ where $c_j = a_j$. We define $B' = (a_1, \dots, a_{j-1}, c_j, \dots, c_k)^{tr}$ and $C' = (c_1, \dots, c_{j-1}, b_j, \dots, b_k)^{tr}$. Note that $B'$ and $C'$ are valid tableaux because $a_{j-1} < a_j = c_j$ and $c_{j-1} < c_j = a_j < b_j$. So we may apply the following relation:

\[
\begin{tabular}{|cc|}
    \multicolumn{2}{c}{$B \qquad C$} \\
    \hline
    $a_1$       & $c_1$     \\
    $\vdots$    & $\vdots$  \\
    $a_{j-1}$   & $c_{j-1}$ \\
    $\textcolor{blue}{b_j}$       & $\textcolor{blue}{c_j}$     \\
    $\textcolor{blue}{\vdots}$    & $\textcolor{blue}{\vdots}$  \\
    $\textcolor{blue}{b_k}$       & $\textcolor{blue}{c_k}$     \\
    \hline
\end{tabular}
=
\begin{tabular}{|cc|}
    \multicolumn{2}{c}{$B' \qquad C'$} \\
    \hline
    $a_1$       & $c_1$     \\
    $\vdots$    & $\vdots$  \\
    $a_{j-1}$   & $c_{j-1}$ \\
    $c_j$       & $b_j$     \\
    $\vdots$    & $\vdots$  \\
    $c_k$       & $b_k$     \\
    \hline
\end{tabular}\ .
\]

After applying this relation, we have reduced the number of differences in columns $A$ and $B$. So by induction, we can find a sequence of swaps after which the leftmost column of $T$ and $T'$ are equal. \qed

\end{proof}

\begin{example}\label{example:swap1}
Consider $\Gr(3,6)$ with block diagonal matching field $B_{3} = (123 | 456)$. Each tableau $T$ below is partitioned with $T_X$ on the left and $T_Y$ on the right, where each part is in semi-standard form. We show that $T_1$ is quadratically equivalent to $T_3$ with the following sequence of quadratic relations. Firstly, we may swap $2$ and $3$ in the second row of $T_1$ to obtain $T_2$ and then swap $5$ and $6$ from the first and second columns of $T_2$ to obtain $T_3$.
\[
T_1 = 
\begin{tabular}{|c|cc|}
    \hline
    1                   & 4 & 5                     \\
    {\bf\textcolor{blue}{3}} & 1 & {\bf\textcolor{blue}{2}}   \\
    5                   & 6 & 6                     \\
    \hline
\end{tabular}
\qquad T_2 =
\begin{tabular}{|c|cc|}
    \hline
    1   & 4 & 5 \\
    2   & 1 & 3 \\
    {\bf\textcolor{blue}{5}}   & {\bf\textcolor{blue}{6}} & 6 \\
    \hline
 \end{tabular}
\qquad T_3 =
\begin{tabular}{|c|cc|}
    \hline
    1   & 4 & 5 \\
    2   & 1 & 3 \\
    6   & 5 & 6 \\
    \hline
\end{tabular}\ 
\]
\end{example}

\begin{lemma}\label{lem:swaps_case2}
Suppose that $T$ and $T'$ are row-wise equal tableaux and each contain columns of type $0$ and $1$ only. Then there exist tableaux $S$ and $S'$ quadratically equivalent to $T$ and $T'$, respectively, such that the first two rows of $S$ and $S'$ are identical.
\end{lemma}

\begin{proof}
Since the second row entries of $T_X$ and $T'_X$ lie in $B_{\ell,2}$ and the second row entries of $T_Y$ and $T_Y'$ lie in $B_{\ell,1}$, by semi-standardness of the tableaux, it follows that the second row of $T$ is identical to the second row of $T'$. Note that, if the first row of $T_Y$ is equal to the first row of $T'_Y$, then the first row of $T_X$ is equal to the first row of $T'_X$ and vice versa because $T_X, T'_X, T_Y$ and $T'_Y$ are in semi-standard form. We proceed by induction on the number of differences in the first row of $T$ and $T'$.

Suppose the first row of $T$ is not equal to the first row of $T'$. Let us write the first row of $T_Y$ as $(\alpha_1, \alpha_2, \dots, \alpha_t)$ for some $t \ge 1$, and similarly write $(\beta_1, \beta_2, \dots, \beta_t)$ for the first row of $T'_Y$. By assumption there exists $1 \le i \le t$ such that $\alpha_i \neq \beta_i$. Let $i$ be the smallest such index. Let $A = (a_1, \dots, a_k)^{tr}$ be the column of $T_Y$ whose first entry is $\alpha_i=a_1$ and let $B = (b_1, \dots, b_k)^{tr}$ be the column of $T'_Y$ whose first entry is $\beta_i=b_1$. Since $A$ and $B$ are the $i^{\rm th}$ column of $T_Y$ and $T'_Y$, respectively, and the second rows are the same, we have that $b_2 = a_2$.

Assume without loss of generality that $a_1 < b_1$. Since the tableaux are semi-standard and row-wise
equal, there is a column $C = (c_1, \dots, c_k)^{tr}$ in $T'_X$ where $c_1 = a_1$. We distinguish the following cases:

\textbf{Case 1.} $b_1 < c_2$. Then we may swap $a_1$ and $b_1$ in $T'$. As a result columns $A$ and $B$ have the same entry in the first and second row and so we have reduced the number of differences in the first row.

\textbf{Case 2.} $b_1 \ge c_2$. Then there is a column $D = (d_1, \dots, d_k)^{tr}$ in $T_X$ such that $d_2 = c_2$ and $D$ is in the same position in $T$ as column $C$ is in $T'$. Now we take cases on $d_1$ and $a_3$.

\textbf{Case 2.i.} $d_1 < a_3$. Then we may swap $d_1$ and $a_1$ in $T$. As a result, columns $C$ and $D$ will have the same entries in the first and second row and so we have reduced the number of differences in the first row. 

\textbf{Case 2.ii.} $d_1 \ge a_3$. Then we have $b_1 \ge c_2 > d_1 \ge a_3 > a_1$. In particular $d_1 < b_1$. Hence, $d_1$ appears in the first row for some column in $T'_X$. Call this column $E = (d_1, e_2, \dots, e_k)^{tr}$; see Figure~\ref{fig:tableau}. Since $a_1 < d_1$ and $T'_X$ is in semi-standard form, column $E$ is on the right side of column $C$. Therefore, we may swap the entries $d_1$ and $a_1$ in $T'$. 
So the first two entries of columns $C$ and $D$ are the same and we have reduced the number of differences in the first row.
\begin{figure}
    \centering
    \begin{minipage}{.2\textwidth}
        \begin{tabular}{|c|c|}
            \multicolumn{1}{c}{$D$} & \multicolumn{1}{c}{$A$} \\
            \hline
            $d_1$       & $a_1$     \\
            $c_2$       & $a_2$  \\
            $d_3$       & $a_3$ \\
            $\vdots$    & $\vdots$  \\
            $d_k$       & $a_k$     \\
            \hline
        \end{tabular}
    \end{minipage}
    \begin{minipage}{.2\textwidth}
        \begin{tabular}{|cc|c|}
            \multicolumn{1}{c}{$C$} & \multicolumn{1}{c}{$E$} & \multicolumn{1}{c}{$B$} \\
            \hline
            $\textcolor{blue}{a_1}$       & $\textcolor{blue}{d_1}$     & $b_1$     \\
            $c_2$       & $e_2$     & $a_2$  \\
            $c_3$       & $e_2$     & $b_3$ \\
            $\vdots$    & $\vdots$  & $\vdots$  \\
            $c_k$       & $e_k$     & $b_k$     \\
            \hline
        \end{tabular}
    \end{minipage}
    \caption{Depiction of columns $A$, $D$ from $T$ along with columns $B$, $C$, $E$ from $T'$ (Case 2.ii in Lemma~\ref{lem:swaps_case2}).}
    \label{fig:tableau}
\end{figure}
Therefore, by induction on the number of differences in the first row of $T$ and $T'$, we can apply quadratic relations to $T$ and $T'$ to ensure the first two rows are identical. \end{proof}

\begin{example}\label{example:swap2}
For Lemma~\ref{lem:swaps_case2}, Case 1, consider $\Gr(3,6)$ with the block diagonal matching field $\B_1$. Consider the tableaux:
\[
T = 
\begin{tabular}{|c|cc|}
    \multicolumn{1}{c}{} & \multicolumn{1}{c}{$A$} & \multicolumn{1}{c}{} \\
    \hline
    3   & 2 & 4 \\
    4   & 1 & 1 \\
    6   & 5 & 5 \\
    \hline
\end{tabular} \, ,
\quad T' =
\begin{tabular}{|c|cc|}
    \multicolumn{1}{c}{$C$} & \multicolumn{1}{c}{$B$} & \multicolumn{1}{c}{} \\
    \hline
    {\bf\textcolor{blue}{2}}   & {\bf\textcolor{blue}{3}} & 4 \\
    4   & 1 & 1 \\
    5   & 5 & 6 \\
    \hline
\end{tabular}\ .
\]
We let $A = (2,1,5)^{tr}$ in $T$ and $B = (3,1,5)^{tr}$ in $T'$. Then, following the proof, let $C = (2,4,5)^{tr}$ in $T'$. Since $b_1 = 3 < 4 = c_2$ we may swap $2$ and $3$ in the first row of $T'$. This makes the first two rows of $T$ and $T'$ identical.

\medskip

For Lemma~\ref{lem:swaps_case2}, Case 2.i, consider $\Gr(3,7)$ with the block diagonal matching field $\B_1$. Consider the tableaux $T$ and $T'$: 
\[
T = 
\begin{tabular}{|cc|c|}
    \multicolumn{1}{c}{$D$} & \multicolumn{1}{c}{} & \multicolumn{1}{c}{$A$} \\
    \hline
    {\bf\textcolor{blue}{3}}   & 4 & {\bf\textcolor{blue}{2}} \\
    4   & 5 & 1 \\
    6   & 7 & 5 \\
    \hline
\end{tabular} \, ,
\quad T' =
\begin{tabular}{|cc|c|}
    \multicolumn{1}{c}{$C$} & \multicolumn{1}{c}{} & \multicolumn{1}{c}{$B$} \\
    \hline
    2   & 3 & 4 \\
    4   & 5 & 1 \\
    5   & 6 & 7 \\
    \hline
\end{tabular}\, ,
\quad \tilde{T} = 
\begin{tabular}{|cc|c|}
    \multicolumn{1}{c}{$\tilde{D}$} & \multicolumn{1}{c}{} & \multicolumn{1}{c}{$\tilde{A}$} \\
    \hline
    2   & {\bf\textcolor{blue}{4}} & {\bf\textcolor{blue}{3}} \\
    4   & 5 & 1 \\
    6   & 7 & 5 \\
    \hline
\end{tabular}\ .
\]
Following the proof, since $d_1 = 3 < 5 = a_3$ we may swap $d_1$ and $a_3$ to obtain the tableau $\tilde{T}$ with altered columns $\tilde{D}$ and $\tilde{A}$.
Now observe that the first rows of $\tilde{T}$ and $T'$ differ in only two positions, namely the second and third columns.

\medskip

Finally for Lemma~\ref{lem:swaps_case2}, Case 2.ii consider $\Gr(4,8)$ with the block diagonal matching field $\B_1$. Consider the tableaux $T$ and $T'$:
\[
T = 
\begin{tabular}{|ccc|c|}
    \multicolumn{1}{c}{} & \multicolumn{1}{c}{$D$} & \multicolumn{1}{c}{} & \multicolumn{1}{c}{$A$} \\
    \hline
    2   & 4 & 5 & 3 \\
    3   & 5 & 6 & 1 \\
    6   & 6 & 7 & 4 \\
    7   & 8 & 8 & 5 \\
    \hline
\end{tabular} \, ,
\quad T' =
\begin{tabular}{|ccc|c|}
    \multicolumn{1}{c}{} & \multicolumn{1}{c}{$C$} & \multicolumn{1}{c}{$E$} & \multicolumn{1}{c}{$B$} \\
    \hline
    2   & {\bf\textcolor{blue}{3}} & {\bf\textcolor{blue}{4}} & 5 \\
    3   & 5 & 6 & 1 \\
    4   & 6 & 7 & 6 \\
    5   & 7 & 8 & 8 \\
    \hline
\end{tabular}\, ,
\quad \tilde{T'} =
\begin{tabular}{|ccc|c|}
    \multicolumn{1}{c}{} & \multicolumn{1}{c}{$\tilde{C}$} & \multicolumn{1}{c}{$\tilde{E}$} & \multicolumn{1}{c}{$B$} \\
    \hline
    2   & 4 & {\bf\textcolor{blue}{3}} & {\bf\textcolor{blue}{5}} \\
    3   & 5 & 6 & 1 \\
    4   & 6 & 7 & 6 \\
    5   & 7 & 8 & 8 \\
    \hline
\end{tabular}\ .
\]
In this case, we have $d_1 \ge a_3$ so we swap $c_1=3$ and $e_1=4$ in $T'$ to obtain the tableau $\tilde{T'}$.
Notice that the first rows of $T$ and $T'$ differ in last three columns whereas the first rows of $T$ and $\tilde{T'}$ differ only in the third and fourth columns.

Let us take this example a little further. To make the first two rows identical, we proceed by Case 2.i and swap $3$ and $5$ in the first row of $\tilde{T'}$. The manipulation of these tableaux is continued in Example~\ref{example:quad_gen}.
\end{example}

\subsection{\bf SAGBI basis from matching field tableaux.}\label{sec:core}
In this section, for each block diagonal matching field $\BLambda$, we will construct a collection of two-column tableaux such that every two-column tableau is row-wise equal to exactly one tableau in the collection. Then we will show that this collection is in bijection with two-column tableaux in semi-standard form. We have already seen in Example~\ref{example:semi_std_basis} that the set of semi-standard tableaux gives a SAGBI basis for the Pl\"ucker algebra with respect to the weight vector arisen from the diagonal matching field, hence this construction gives a strict generalization.
We first recall the definition of {\bf SAGBI basis}\footnote{SAGBI stands for Subalgebra Analogue to Gr\"obner Bases for Ideals. The definition of Khovanskii bases, from \cite{KM16}, generalizes the notion of a SAGBI bases to non-polynomial algebras.} from \cite{robbiano1990subalgebra} in our setting. \begin{definition}\label{sagbi}
Let $\mathcal{A}_{k,n}$ be the Pl\"ucker algebra of $\Gr(k,n)$ and let $A_{\ell} =\mathbb{K}[\inwb(P_I)]_{I \in \mathbf{I}_{k,n}}$ be the \emph{algebra} of $\BLambda$. The set of Pl\"ucker forms $\{P_I\}_{I\in \mathbf{I}_{k,n}}\subset \mathbb{K}[x_{ij}]$ is a SAGBI basis for $\mathcal{A}_{k,n}$ with respect to the weight vector $\wb_\ell$ if and only if for each $I\in \mathbf{I}_{k,n}$, the initial form $\inwb(P_I) $ is a monomial and
$\inwb(\mathcal{A}_{k,n})=A_\ell$. Here, $\wb_\ell$ is the weight vector induced by the matrix $M_\ell$. 
\end{definition}

Notice that the algebra $A_{\ell}$ has the standard grading and any monomial $P_{I_1}P_{I_2} \dots P_{I_t} \in A_{\ell}$ is identified with a tableau of size $k$ by $t$ denoted by $T_{I_1 I_2 \dots I_t}$.
We show that the subspace $[A_{\ell}]_2$ of $A_{\ell}$ spanned by the monomials of degree two has a basis
which is in a bijection with the set of semi-standard tableaux with two columns.

\begin{definition}\label{def:basis}
Fix $\Gr(k,n)$ and a block diagonal matching field $\BLambda$. We define $\Tc_{\ell}$ to be the collection of all tableaux $T$ which follow. We partition $\Tc_{\ell}$ into types each of which is described below. We also define a map $S$ taking each tableau $T \in \Tc_{\ell}$ to a semi-standard tableau $S(T)$. Note that $S(T)$ does not necessarily lie in $\Tc_{\ell}$. We write,
\[
I = \{i_1 < i_2 < \dots < i_k\} \quad \text{and} \quad
J = \{j_1 < j_2 < \dots < j_k\}.
\]

\textbf{Type 1.}\label{type:1}
\[
T = 
\begin{tabular}{|c|c|}
    \hline
    $i_1$     & $j_1$       \\
    $i_2$     & $j_2$       \\
    $\vdots$  & $\vdots$    \\
    $i_k$     & $j_k$       \\
    \hline
\end{tabular} \quad
\begin{tabular}{l}
    $i_1 \le j_1, i_2 \le j_2, \dots, i_k \le j_k$,
\end{tabular} \quad
S(T) = 
\begin{tabular}{|c|c|}
    \hline
    $i_1$     & $j_1$       \\
    $i_2$     & $j_2$       \\
    $\vdots$  & $\vdots$    \\
    $i_k$     & $j_k$       \\
    \hline
\end{tabular}\ .
\]

\medskip

\textbf{Type 2.}\label{type:2}
\[
T = 
\begin{tabular}{|c|c|}
    \hline
    $i_2$     & $j_2$       \\
    $i_1$     & $j_1$       \\
    $i_3$     & $j_3$       \\
    $\vdots$  & $\vdots$    \\
    $i_k$     & $j_k$       \\
    \hline
\end{tabular} \quad
\begin{tabular}{l}
    $i_1 \le j_1, i_2 \le j_2, \dots, i_k \le j_k$,
\end{tabular}\quad
S(T) = 
\begin{tabular}{|c|c|}
    \hline
    $i_1$     & $j_1$       \\
    $i_2$     & $j_2$       \\
    $i_3$     & $j_3$       \\
    $\vdots$  & $\vdots$    \\
    $i_k$     & $j_k$       \\
    \hline
\end{tabular}.
\]

\medskip

\textbf{Type 3A.}\label{type:3A}
\[
T =
\begin{tabular}{|c|c|}
    \hline
    $i_2$     & $j_1$       \\
    $i_1$     & $j_2$       \\
    $i_3$     & $j_3$       \\
    $\vdots$  & $\vdots$    \\
    $i_k$     & $j_k$       \\
    \hline
\end{tabular} \quad
\begin{tabular}{l}
    $i_1 \in B_{\ell,1}$,\\
    $i_2, \dots, i_k, j_1 \dots, j_k \in B_{\ell,2}$,\\
    $i_2 \le j_1$,\\
    $i_3 \le j_3, \dots, i_k \le j_k$,
\end{tabular} \quad
S(T) = 
\begin{tabular}{|c|c|}
    \hline
    $i_1$     & $j_1$       \\
    $i_2$     & $j_2$       \\
    $i_3$     & $j_3$       \\
    $\vdots$  & $\vdots$    \\
    $i_k$     & $j_k$       \\
    \hline
\end{tabular}\ .
\]

\medskip

\textbf{Type 3B(r).}\label{type:3Br}
\[
T = 
\begin{tabular}{|c|c|}
    \hline
    $i_2$           & $j_1$       \\
    $i_1$           & $j_2$       \\
    $i_3$           & $j_3$       \\
    $\vdots$        & $\vdots$    \\
    $i_{r-1}$       & $j_{r-1}$   \\
    $i_{r}$         & $j_{r}$   \\
    $i_{r+1}$       & $j_{r+1}$   \\
    $\vdots$        & $\vdots$    \\
    $i_k$           & $j_k$       \\
    \hline
\end{tabular} \quad
\begin{tabular}{l}
    $i_1 \in B_{\ell,1}$, \\
    $i_2, \dots, i_k, j_1 \dots, j_k \in B_{\ell,2}$,\\
    $j_1 < j_2 \le i_2$,\\
    $r = \min\{t \ge 2 :j_{t+1} > i_{t}\}$,\\
    $i_3 > j_3, \dots, i_r > j_r$,  \\
    $i_{r+1} \le j_{r+1}, \dots, i_k \le j_k$, \\
\end{tabular}\quad
S(T) = 
\begin{tabular}{|c|c|}
    \hline
    $i_1$           & $j_1$       \\
    $j_2$           & $i_2$       \\
    $j_3$           & $i_3$       \\
    $\vdots$        & $\vdots$    \\
    $j_{r-1}$       & $i_{r-1}$   \\
    $j_{r}$         & $i_{r}$   \\
    $i_{r+1}$       & $j_{r+1}$   \\
    $\vdots$        & $\vdots$    \\
    $i_k$           & $j_k$       \\
    \hline
\end{tabular}\ .
\]

\medskip

\textbf{Type 3C(s).}\label{type:3Cs}
\[
T = 
\begin{tabular}{|c|c|}
    \hline
    $i_2$           & $j_1$       \\
    $i_1$           & $j_2$       \\
    $i_3$           & $j_3$       \\
    $\vdots$        & $\vdots$    \\
    $i_{s}$         & $j_{s}$     \\
    $i_{s+1}$       & $j_{s+1}$     \\
    $\vdots$        & $\vdots$    \\
    $i_k$           & $j_k$       \\
    \hline
\end{tabular} \quad
\begin{tabular}{l}
    $i_1, j_1, j_2, \dots, j_s \in B_{\ell,1}$,\\
    $i_2, \dots, i_k, j_{s+1} \dots, j_k \in B_{\ell,2}$,\\
    $i_1 \le j_1 < j_2$,\\
    $i_{s+1} \le j_{s+1}, \dots, i_k \le j_k$,
\end{tabular} \quad
S(T) = 
\begin{tabular}{|c|c|}
    \hline
    $i_1$           & $j_1$       \\
    $j_2$           & $i_2$       \\
    $j_3$           & $i_3$       \\
    $\vdots$        & $\vdots$    \\
    $j_{s}$         & $i_{s}$     \\
    $i_{s+1}$       & $j_{s+1}$     \\
    $\vdots$        & $\vdots$    \\
    $i_k$           & $j_k$       \\
    \hline
\end{tabular}\ .
\]

\medskip

\textbf{Type 3D(s).}\label{type:3Ds}
\[
T = 
\begin{tabular}{|c|c|}
    \hline
    $i_2$           & $j_1$       \\
    $i_1$           & $j_2$       \\
    $i_3$           & $j_3$       \\
    $\vdots$        & $\vdots$    \\
    $i_{s}$         & $j_{s}$     \\
    $i_{s+1}$       & $j_{s+1}$     \\
    $\vdots$        & $\vdots$    \\
    $i_k$           & $j_k$       \\
    \hline
\end{tabular} \quad
\begin{tabular}{l}
    $i_1, j_1, j_2, \dots, j_s \in B_{\ell,1}$,\\
    $i_2, \dots, i_k, j_{s+1} \dots, j_k \in B_{\ell,2}$,\\
    $i_1 \ge j_2$,\\
    $i_{s+1} \le j_{s+1}, \dots, i_k \le j_k$,
\end{tabular} \quad
S(T) = 
\begin{tabular}{|c|c|}
    \hline
    $i_1$           & $j_1$       \\
    $j_2$           & $i_2$       \\
    $j_3$           & $i_3$       \\
    $\vdots$        & $\vdots$    \\
    $j_{s}$         & $i_{s}$     \\
    $i_{s+1}$       & $j_{s+1}$     \\
    $\vdots$        & $\vdots$    \\
    $i_k$           & $j_k$       \\
    \hline
\end{tabular}\ .
\]
\end{definition}

\begin{lemma}\label{lem:basis_span}
The tableaux $\Tc_{\ell}$ span $[A_{\ell}]_2$.
\end{lemma}

\begin{proof}
Let $I = \{i_1 < i_2 < \dots < i_k\}$ and $J = \{j_1 < j_2 < \dots < j_k\}$ be two $k$-subsets of $[n]$. Recall the notation for tableaux in terms of their columns defined on Page~\pageref{notation:tableau}. We show that there is a tableau $T_{I' J'} \in \Tc_{\ell}$ which is row-wise equal to $T_{IJ}$. We take cases on $I$ and $J$.
\begin{figure}
    \centering
    \begin{tabular}{|c|c|}
        \multicolumn{2}{c}{Case 1} \\
        \hline
        $i_1$     & $j_1$       \\
        $i_2$     & $j_2$       \\
        $i_3$     & $j_3$       \\
        $\vdots$  & $\vdots$    \\
        $i_k$     & $j_k$       \\
        \hline 
    \end{tabular} \qquad
    \begin{tabular}{|c|c|}
        \multicolumn{2}{c}{Case 2} \\
        \hline
        $i_2$     & $j_2$       \\
        $i_1$     & $j_1$       \\
        $i_3$     & $j_3$       \\
        $\vdots$  & $\vdots$    \\
        $i_k$     & $j_k$       \\
        \hline
    \end{tabular} \qquad
    \begin{tabular}{|c|c|}
    \multicolumn{2}{c}{Case 3} \\
    \hline
    $i_2$     & $j_1$       \\
    $i_1$     & $j_2$       \\
    $i_3$     & $j_3$       \\
    $\vdots$  & $\vdots$    \\
    $i_k$     & $j_k$       \\
    \hline
\end{tabular}
    
    \caption{Tableaux $T_{IJ}$ in Lemma~\ref{lem:basis_span} for cases 1, 2 and 3.}
    \label{tab:tij_span}
\end{figure}

\textbf{Case 1.} $\BLambda(I) = \BLambda(J) = id$. The entries of $T_{IJ}$ can be seen in Case 1 of Figure~\ref{tab:tij_span}.
Let
\[
I' = \{\min\{i_1, j_1\} < \dots < \min\{i_k, j_k\} \}, \quad 
J' = \{\max\{i_1, j_1\} < \dots < \max\{i_k, j_k\} \}.
\]
Then $T_{I' J'}$ and $T_{I J}$ are row-wise equal and $T_{I' J'}$ is a tableau of type 1 in $\Tc_{\ell}$.

\textbf{Case 2.} $\BLambda(I) = \BLambda(J) = (12)$. The entries of $T_{IJ}$ can be seen in Case 2 of Figure~\ref{tab:tij_span}. Similarly to case 1, let
\[
I' = \{\min\{i_1, j_1\} < \dots < \min\{i_k, j_k\} \}, \quad 
J' = \{\max\{i_1, j_1\} < \dots < \max\{i_k, j_k\} \}.
\]

Then $T_{I' J'}$ and $T_{I J}$ are row-wise equal and $T_{I' J'}$ is a tableau of type 2 in $\Tc_{\ell}$.

\textbf{Case 3.} $\BLambda(I) \neq \BLambda(J)$. Without loss of generality we assume $\BLambda(I) = (12)$ and $\BLambda(J) = id$. The entries of $T_{IJ}$ can be seen in Case 3 of Figure~\ref{tab:tij_span}.
Since $\BLambda(I) = (12)$ it follows that $i_1 \in B_{\ell,1}$ and $i_2, \dots, i_k \in B_{\ell,2}$. We now take cases on the type of $J$.

\textbf{Case 3a.} 
The type of $J$ is $0$ and $i_2 < j_2$. Let
\[
I' = \{i_1 < \min(i_2, j_1) < \min(i_3, j_3) < \dots < \min(i_k, j_k) \},
\]
\[
J' = \{\max(i_2, j_1) < j_2 < \max(i_3, j_3) < \dots < \max(i_k, j_k) \}.
\]
So we have $T_{IJ}$ is row-wise equal to $T_{I'J'}$ and $T_{I'J'}$ is a tableau of type 3A in $\Tc_{\ell}$.

\textbf{Case 3b.} 
The type of $J$ is $0$ and $i_2 \ge j_2$. Let $r = \min\{t \ge 2 : j_{t+1} > i_t \}$. Let
\[
I' = \{i_1 < i_2 < i_3 < \dots < i_r < \min(i_{r+1}, j_{r+1}) < \dots < \min(i_k, j_k) \},
\]
\[
J' = \{j_1 < j_2 < j_3 < \dots < j_r < \max(i_{r+1}, j_{r+1}) < \dots < \max(i_k, j_k) \}.
\]
Note that for each $t \in \{3, \dots, r\}$ we have that $j_t \le i_{t-1} < i_t$. So $T_{I'J'}$ is a tableau of type 3B(r) in $\Tc_{\ell}$ and $T_{IJ}$ is row-wise equal to $T_{I'J'}$.

\textbf{Case 3c.} 
The type of $J$ is $s \ge 2$ and $i_1 \le j_1$. Then let
\[
I' = \{i_1 < i_2 < i_3 < \dots < i_s < \min(i_{s+1}, j_{s+1}) < \dots < \min(i_k, j_k) \},
\]
\[
J' = \{j_1 < j_2 < j_3 < \dots < j_s < \max(i_{s+1}, j_{s+1}) < \dots < \max(i_k, j_k) \}.
\]
So we have $T_{IJ}$ is row-wise equal to $T_{I'J'}$ and $T_{I'J'}$ is a tableau of type 3C(s) in $\Tc_{\ell}$.

\textbf{Case 3d.} 
The type of $J$ is $s \ge 2$ and $i_1 > j_1$. Now if $i_1 < j_2$ then we may apply a quadratic relation to swap $i_1$ and $j_2$. So without loss of generality we may assume that $i_1 \ge j_2$. Let

\[
I' = \{i_1 < i_2 < i_3 < \dots < i_s < \min(i_{s+1}, j_{s+1}) < \dots < \min(i_k, j_k) \},
\]
\[
J' = \{j_1 < j_2 < j_3 < \dots < j_s < \max(i_{s+1}, j_{s+1}) < \dots < \max(i_k, j_k) \}.
\]
So we have $T_{IJ}$ is row-wise equal to $T_{I'J'}$ and $T_{I'J'}$ is a tableau of type 3D(s) in $\Tc_{\ell}$. And so we have shown that $\Tc_{\ell}$ spans $[A_{\ell}]_2$. \qed

\end{proof}

A visual representation of the proof is shown in Figure~\ref{flowchart:gr_span_pf} with key given in Figure~\ref{flowchart:gr_key}. In particular Figure~\ref{flowchart:gr_span_pf} illustrates each type of tableau appearing in $\Tc_{\ell}$.

\begin{figure}
    \centering
    \includegraphics[scale = 0.8]{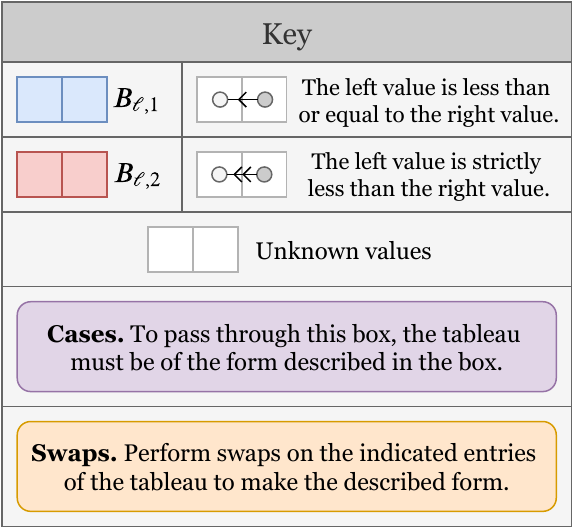}
    \caption{The key for Figure~\ref{flowchart:gr_span_pf}. We assume that there is a fixed matching field $\BLambda$. The blue and red boxes represent entries of the tableau which belong to $B_{\ell,1}$ and $B_{\ell,2}$ respectively. Arrows between boxes represent inequalities between their respective entries. An arrow with a double head indicates that this inequality is strict. Purple boxes are requirements for the tableau and in particular do not alter the tableau. Orange boxes alter the tableau by, possibly, performing swaps between indicated entries until the tableau is of the described form.}
    \label{flowchart:gr_key}
\end{figure}

\begin{figure}
    \centering
    \includegraphics[scale = 0.8]{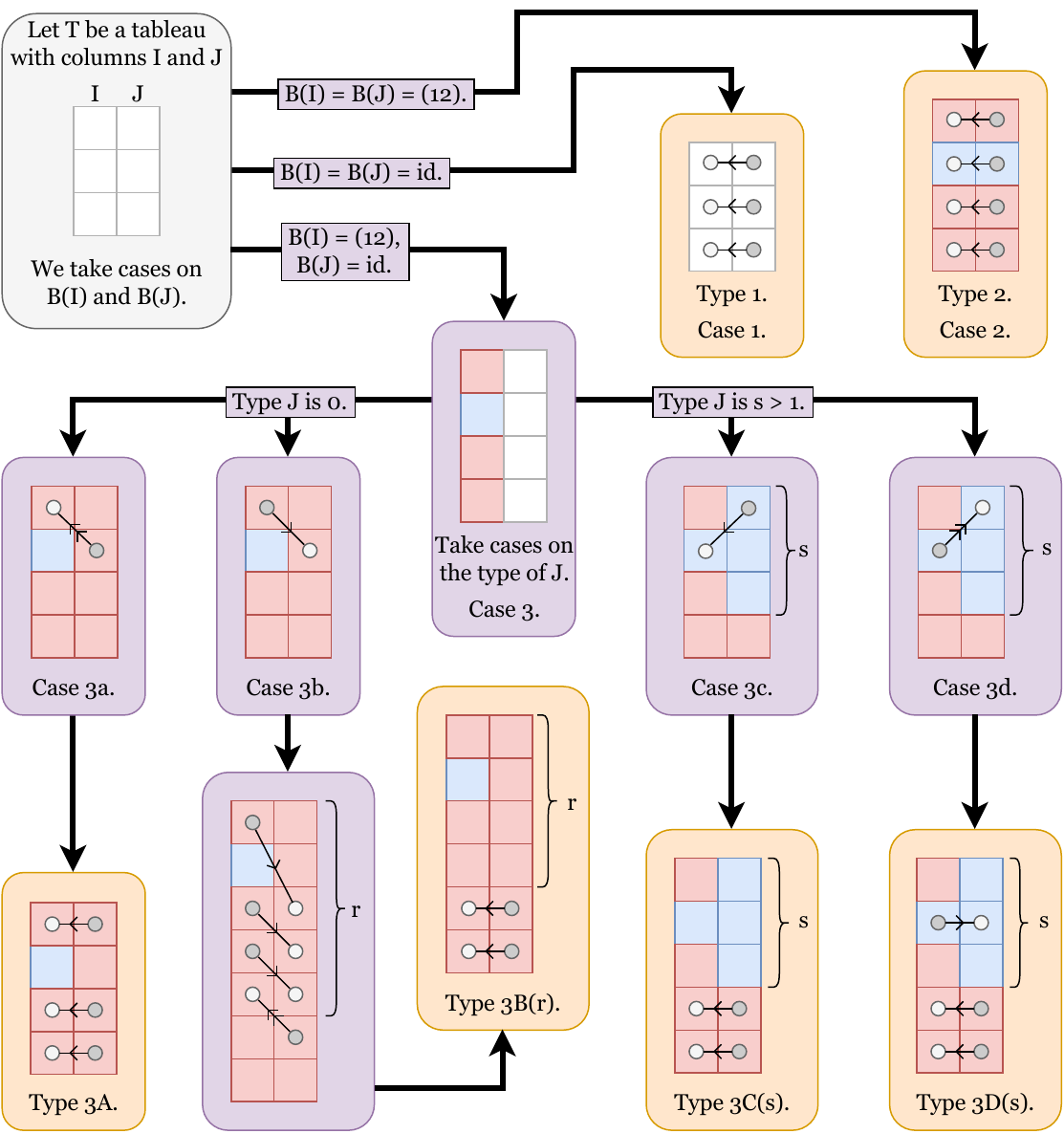}
    \caption{A depiction of the proof of Lemma~\ref{lem:basis_span}. For notation we write $\B$ for the matching field $\BLambda$ which is fixed throughout the proof. The type of column can be seen as the number of blue boxes appearing in that column. The orange boxes give a visual representation of the different types of tableaux in $\Tc_{\ell}$, see Definition~\ref{def:basis}. Note that the swaps that occur in the orange boxes interchange entries in the same row and each column is ordered with respect to the matching field.}
    \label{flowchart:gr_span_pf}
\end{figure}

\begin{lemma}\label{lem:basis_indep}
The tableaux $\Tc_{\ell}$ are linearly independent in $[A_{\ell}]_2$.
\end{lemma}

\begin{proof}
Consider the map $\phi_{\BLambda}$ from Equation~\eqref{eqn:monomialmap} in  Page~\pageref{eqn:monomialmap}. Since $\Ker(\phi_{\BLambda})$ is generated by binomials, it suffices to show that if $T_{IJ}$ and $T_{I',J'} \in \Tc_{\ell}$ are row-wise equal then they are equal up to reordering the columns. 
 
Recall that the type of a column $I$ in $T_{IJ}$ is $|I \cap B_{\ell,1}|$ and if $T_{IJ}$ and $T_{I'J'}$ are row-wise equal then their columns must have the same type up to reordering. So we may assume that
$|I \cap B_{\ell,1}| = |I' \cap B_{\ell,1}|$ and $s = |J \cap B_{\ell,1}| = |J' \cap B_{\ell,1}|$.
Now we take cases on the type of $T_{IJ}$ in $\Tc_{\ell}$.

\smallskip

\textbf{Type 1.} In this case, we have that both $T_{IJ}$ and $T_{I'J'}$ are in semi-standard form. Since the rows are weakly increasing, it follows that these tableaux must be equal.

\smallskip

\textbf{Type 2.} Again we see that $T_{IJ}$ and $T_{I'J'}$ are equal since the rows are weakly-increasing.

\smallskip

\textbf{Type 3.} We assume that $\BLambda(I) = (12)$ and $\BLambda(J) = id$. Note that the columns $J$ and $J'$ have the same type denoted $s$. If $s = 0$ then $T_{IJ}$ and $T_{I'J'}$ are either of type 3A or 3B, see Page~\pageref{type:3A}. Otherwise $s \ge 2$ and $T_{IJ}$ and $T_{I'J'}$ are either of type 3C or 3D, see Page~\pageref{type:3Cs}. In both cases we write
\[
I = \{i_1 < i_2 < \dots < i_k \}, \quad
J = \{j_1 < j_2 < \dots < j_k \},
\]
\[
I' = \{i'_1 < i'_2 < \dots < i'_k \}, \quad
J' = \{j'_1 < j'_2 < \dots < j'_k \}.
\]

{\bf Case 1.} Let $s = 0$, i.e. $T_{IJ}$ and $T_{I'J'}$ are either of type 3A or 3B. Since $i_1, i'_1$ are the only elements of $I$ and $J$ that belong to $B_{\ell,1}$, it follows that $i_1 = i'_1$. Hence, by row-wise equality of $T_{IJ}$ and $T_{I'J'}$ we have $j_2 = j'_2$.

If $i_2 < j_2$ then we have that $T_{IJ}$ is of type 3A. By row-wise equality of $T_{IJ}$ and $T_{I'J'}$ we have $\{i'_2, j'_1\} = \{i_2, j_1 \}$ and so $i'_2 < j'_2$. Therefore $T_{I'J'}$ is also of type 3A. By definition of type 3A, the rows are weakly increasing and so $T_{IJ}$ and $T_{I'J'}$ are equal.

If, on the other hand, $i_2 \ge j_2$ then $T_{IJ}$ is of type 3B. Note that $j'_1 < j'_2 = j_2$ so $i_2 \neq j'_1$. By row-wise equality of $T_{IJ}$ and $T_{I'J'}$ we deduce that $i_2 = i'_2$ and $j_1 = j'_1$. Therefore $T_{I'J'}$ is also of type 3B. For each $t \in \{2, \dots, r \}$ where $r = \min\{t \ge 2: j_{t+1} > i_t \}$ we show that $i_t = i'_t$. 

Suppose by contradiction that $i_t \neq i'_t$ for some $t \in \{2, \dots, r \}$. We take $t$ to be the smallest such value. Note that $i_2 = i'_2$ so we may assume $3 \le t \le r$. By row-wise equality of $T_{IJ}$ and $T_{I'J'}$, since $i'_t \neq i_t$, we have $i'_t = j_t$. By definition of $r$ we have $j_t \le i_{t-1}$. By assumption $t$ was minimal so $i_{t-1} = i'_{t-1}$. So we have deduced that $i'_t \le i'_{t-1}$, a contradiction. Therefore $i_1 = i'_1, i_2 = i'_2, \dots, i_r = i'_r$. 

Next, we will show that $j'_{r+1} > i'_r$. By row-wise equality of $T_{IJ}$ and $T_{I'J'}$, either $j'_{r+1} = j_{r+1}$ or $j'_{r+1} = i_{r+1}$. So either $j'_{r+1} = j_{r+1} > i_r = i'_r$ or $j'_{r+1} = i_{r+1} > i_{r} = i'_{r}$. So we have shown that $r = \min\{t \ge 2: j'_{t+1} > i'_t \}$. Therefore both $T_{IJ}$ and $T_{I'J'}$ are of type 3B(r). Therefore we have $i_{r+1} \le j_{r+1}, \dots, i_k \le j_k$ and $i'_{r+1} \le j'_{r+1}, \dots, i'_k \le j'_k$. By row-wise equality of $T_{IJ}$ and $T_{I'J'}$ it follows that $i_{r+1} = i'_{r+1}, \dots, i_k = i'_k$. And so, we have shown that $I = I'$, hence $J = J'$ and the tableaux are equal. This completes the proof for the case where $T_{IJ}$ and $T_{I'J'}$ are either of type 3A or 3B.

\smallskip

{\bf Case 2.} Let $s \ge 2$, i.e. $T_{IJ}$ and $T_{I'J'}$ are either of type 3C or 3D. Consider the first row of the tableaux $T_{IJ}$. Since $i_2, i'_2 \in B_{\ell,2}$ and $j_1, j'_1 \in B_{\ell,1}$ it follows that $i_2 = i'_2$ and $j_1 = j'_1$. 

If $i_1 \le j_1$ then $T_{IJ}$ is of type 3C(s) and $i_1 = i'_1$ and $j_2 = j'_2$. So $T_{I'J'}$ is also of type 3C(s). Since $j_3, \dots, j_s,j'_3, \dots, j'_s \in B_{\ell,1}$ and $i_3, \dots, i_s, i'_3, \dots, i'_s \in B_{\ell,2}$ it follows that $i_3 = i'_3, \dots, i_s = i'_s$ and $j_3 = j'_3, \dots, j_s = j'_s$. Since $i_{s+1} \le j_{s+1}, \dots, i_k \le j_k$ and $i'_{s+1} \le j'_{s+1}, \dots, i'_k \le j'_k$ so by row-wise equality of $T_{IJ}$ and $T_{I'J'}$ we deduce that $i_{s+1} = i'_{s+1}, \dots, i_k = i'_k$. And so $T_{IJ}$ and $T_{I'J'}$ are equal.

If $i_1 > j_1$ then we have that $T_{IJ}$ is not of type 3C(s) and so it must be of type 3D(s). Therefore $i_1 \le j_2$ and so both entries of the second row of $T_{I'J'}$ are greater than $j_1$. So $T_{I'J'}$ is also not of type 3C(s) and so it must be of type 3D(s). By the definition of 3D(s) we have $i_1 \le j_2$ and $i'_1 \le j'_2$. Therefore, $i_1 = i'_1$ and $j_2 = j'_2$. Since $j_3, \dots, j_s,j'_3, \dots, j'_s \in B_{\ell,1}$ and $i_3, \dots, i_s, i'_3, \dots, i'_s \in B_{\ell,2}$ it follows that $i_3 = i'_3, \dots, i_s = i'_s$ and $j_3 = j'_3, \dots, j_s = j'_s$. Since $i_{s+1} \le j_{s+1}, \dots, i_k \le j_k$ and $i'_{s+1} \le j'_{s+1}, \dots, i'_k \le j'_k$, by row-wise equality of $T_{IJ}$ and $T_{I'J'}$ we deduce that $i_{s+1} = i'_{s+1}, \dots, i_k = i'_k$. And so $T_{IJ}$ and $T_{I'J'}$ are equal. Hence, the proof is complete for the case when $T_{IJ}$ and $T_{I'J'}$ are either of type 3C or 3D. So we have shown that $\Tc_{\ell}$ is a linearly independent set. \qed

\end{proof}

\begin{lemma}\label{lem:basis_bijection}
The map $T \mapsto S(T)$ is a bijection between $\Tc_{\ell}$ and the set of semi-standard tableaux with two columns and $k$ rows.
\end{lemma}

\begin{proof}
We will show that the inverse to the map $S$ from Definition~\ref{def:basis} exists by constructing the map explicitly. Fix a semi-standard tableau
\[
T = 
\begin{tabular}{|c|c|}
    \hline
    $i_1$           & $j_1$       \\
    $i_2$           & $j_2$       \\
    $i_3$           & $j_3$       \\
    $\vdots$        & $\vdots$    \\
    $i_{k}$         & $j_{k}$     \\
    \hline
\end{tabular}\ .
\]
We now take cases on $i_1, i_2, j_1$ and $j_2$. If we have one of,
\begin{itemize}
    \item $i_1, i_2, j_1, j_2 \in B_{\ell,1}$,
    \item $i_1, i_2 \in B_{\ell,1}$ and $j_1, j_2  \in B_{\ell,2}$,
    \item $i_1, i_2, j_1, j_2 \in B_{\ell,2}$,
\end{itemize} 
then we have that $S^{-1}(T) = T$ is a tableau of type 1.

\medskip

If $i_1, j_1 \in B_{\ell,1}$ and $i_2, j_2 \in B_{\ell,2}$, then we have that $S^{-1}(T) = T'$ is a tableau of type 2, where
\[
T' = 
\begin{tabular}{|c|c|}
    \hline
    $i_2$           & $j_2$       \\
    $i_1$           & $j_1$       \\
    $i_3$           & $j_3$       \\
    $\vdots$        & $\vdots$    \\
    $i_{k}$         & $j_{k}$     \\
    \hline
\end{tabular}\ .
\]

\medskip

If $i_1 \in B_{\ell,1}$ and $j_1, i_2, j_2 \in B_{\ell,2}$, then there are two cases, either $i_2 \le j_1$ or $i_2 > j_1$.

\textbf{Case 1.} $i_2 \le j_1$. Then $S^{-1}(T) = T_1$ in Figure~\ref{tab:S_inv} is a tableau of type 3A.

\textbf{Case 2.} $i_2 > j_1$. Let $r = \min\{ t \ge 2: i_{t+1} > j_t \}$. Then $S^{-1}(T) = T_2$ in Figure~\ref{tab:S_inv} is a tableau of type 3B(r).

\medskip

If $i_1, j_1, i_2 \in B_{\ell,1}$ and $j_2 \in B_{\ell,2}$, then there are two cases, either $j_1 < i_2$ or $j_1 \ge i_3$.

\textbf{Case 1.} $j_1 < i_2$. Then $S^{-1}(T) = T_3$ in Figure~\ref{tab:S_inv} is a tableau of type 3C(s), where $s = |I \cap B_{\ell,1}|$.

\textbf{Case 2.} $j_1 \ge i_2$. Then $S^{-1}(T)=T_4$ in Figure~\ref{tab:S_inv} is a tableau of type 3D(s), where $s = |I \cap B_{\ell,1}|$.
\begin{figure} 
    \centering
    $T_1 = $
    \begin{tabular}{|c|c|}
        \hline
        $i_2$           & $j_1$       \\
        $i_1$           & $j_2$       \\
        $i_3$           & $j_3$       \\
        $\vdots$        & $\vdots$    \\
        $i_{k}$         & $j_{k}$     \\
        \hline
    \end{tabular}
    \quad
    $T_2 = $
    \begin{tabular}{|c|c|}
        \hline
        $j_2$           & $j_1$       \\
        $i_1$           & $i_2$       \\
        $j_3$           & $i_3$       \\
        $\vdots$        & $\vdots$    \\
        $j_{r}$         & $i_{r}$     \\
        $i_{r+1}$       & $j_{r+1}$   \\
        $\vdots$        & $\vdots$    \\
        $i_k$           & $j_k$       \\
        \hline
    \end{tabular}
    \quad
    $T_3 = $
    \begin{tabular}{|c|c|}
        \hline
        $j_2$           & $j_1$       \\
        $i_1$           & $i_2$       \\
        $j_3$           & $i_3$       \\
        $\vdots$        & $\vdots$    \\
        $j_{s}$         & $i_{s}$     \\
        $i_{s+1}$       & $j_{s+1}$   \\
        $\vdots$        & $\vdots$    \\
        $i_k$           & $j_k$       \\
        \hline
    \end{tabular}
    \quad
    $T_4 = $
    \begin{tabular}{|c|c|}
        \hline
        $j_2$           & $i_1$       \\
        $j_1$           & $i_2$       \\
        $j_3$           & $i_3$       \\
        $\vdots$        & $\vdots$    \\
        $j_{s}$         & $i_{s}$     \\
        $i_{s+1}$       & $j_{s+1}$   \\
        $\vdots$        & $\vdots$    \\
        $i_k$           & $j_k$       \\
        \hline
    \end{tabular}
    \caption{The list of tableaux for defining $S^{-1}$ in Lemma~\ref{lem:basis_bijection}.}
    \label{tab:S_inv}
\end{figure}
So for each $\ell \in \{0, 1, \dots, n-1 \}$ we have shown that $\Tc_{\ell}$ is a basis for $[A_{\ell}]_2$ and there is a bijection between $\Tc_{\ell}$ and semi-standard tableaux with two columns. 
\end{proof}

\section{Toric degenerations of Grassmannians}\label{sec:Grassmannian}
In this section, we state our main results for Grassmannians, where we generalize the results of \cite{KristinFatemeh} from $\Gr(3,n)$ to higher-dimensional Grassmannians.

\begin{theorem}\label{prop:quad}
The ideals of block diagonal matching fields 
are  quadratically generated. 
\end{theorem}
Before giving the proof of Theorem~\ref{prop:quad} we fix our notation. Fix a block diagonal matching field $\BLambda$. Let $\II = \{I_1, \dots, I_r\}$ with $I_i \in \mathbf{I}_{k, n}$ for each $1 \le i \le r$ be a non-empty collection of  $k$-subsets of $[n]$ and consider the tableau $T_{\II}$. Recall, from \S\ref{sec:tableaux_quad} on Page~\pageref{sec:tableaux_quad}, that $T_{\II}$ is written in the form $[T_X | T_Y]$ where $T_X$ and $T_Y$ are tableaux in semi-standard form. The columns of $T_Y$ are the columns of $T_{\II}$ of type $1$ and the remaining columns are contained in $T_X$. 

\begin{proof}
Suppose that $T$ and $T'$ are two row-wise equal tableaux. The statement is equivalent to proving that $T$ and $T'$ differ by a sequence of swaps. We proceed by induction on the number of columns in $T$ and $T'$. We will assume that $T$ and $T'$ contain no identical columns, otherwise we may remove these columns from the subsequent manipulations and apply induction. Note that if $T$ and $T'$ contain two or fewer columns then we are done.

Suppose that the leftmost column of $T$ and $T'$ is of type $i$ where $i \in \{2,3, \dots, k \}$. Then by Lemma~\ref{lem:swaps_case1} we may perform a sequence of swaps to reduce the number of different columns in $T$ and $T'$.

Suppose that $T$ and $T'$ contain only columns of type $0$ and $1$. Then by Lemma~\ref{lem:swaps_case2} we may perform a sequence of swaps to make the first two rows of $T$ and $T'$ identical. Let $A = (a_1,a_2, \dots, a_k)^{tr}$ be the leftmost column of $T_X$ of type $0$. Let $B = (b_1, b_2, \dots, b_k)^{tr}$ be the leftmost column in $T'_X$. By assumption we have that $b_1 = a_1$ and $b_2 = a_2$. Since $A \neq B$, let $i \ge 3$ be the smallest index such that $b_i \neq a_i$. Without loss of generality suppose that $a_i < b_i$. Since $T_X$ and $T'_X$ are in semi-standard form from the third row down, there is a column $C = (c_1, \dots, c_k)^{tr}$ in $T'_Y$ such that $c_i = a_i$. Then we may apply the relation in Figure~\ref{fig:tableau1}.

\begin{figure}
\centering\begin{tabular}{|cc|}
    \multicolumn{2}{c}{$B \qquad C$} \\
    \hline
    $a_1$       & $c_1$     \\
    $\vdots$    & $\vdots$  \\
    $a_{i-1}$   & $c_{i-1}$ \\
    $b_i$       & $a_i$     \\
    $b_{i+1}$   & $c_{i+1}$ \\
    $\vdots$    & $\vdots$  \\
    $b_k$       & $c_k$     \\
    \hline
\end{tabular}
=
\begin{tabular}{|cc|}
    \multicolumn{2}{c}{$B' \qquad C'$} \\
    \hline
    $a_1$       & $c_1$     \\
    $\vdots$    & $\vdots$  \\
    $a_{i-1}$   & $c_{i-1}$ \\
    $a_i$       & $b_i$     \\
    $c_{i+1}$   & $b_{i+1}$ \\
    $\vdots$    & $\vdots$  \\
    $c_k$       & $b_k$     \\
    \hline
\end{tabular}\caption{A quadratic relation}
    \label{fig:tableau1}
\end{figure}

So we have reduced the number of differences in the leftmost column of $T$ and $T'$. So, by induction on the number of differences in the left column, we can find a sequence of swaps which makes the leftmost column of $T$ and $T'$ equal, which completes the proof. \qed

\end{proof}

\begin{example}\label{example:quad_gen}
Following the proof above, if the leftmost column of $T$ and $T'$ is of type $i$ where $i \in \{2, 3, \dots, k\}$ then consider Example~\ref{example:swap1} in \S\ref{sec:tableaux_quad}. This example gives a typical manipulation of the tableaux to reduce the number of differences in the leftmost column.

If the leftmost column of $T$ and $T'$ is of type $0$ or $1$, then Example~\ref{example:swap2} shows how to make the first and second rows of $T$ and $T'$ identical. Continuing the final part of Example~\ref{example:swap2}, we consider the tableaux $T$ and $T'$ below.
\[
T = 
\begin{tabular}{|ccc|c|}
    \multicolumn{1}{c}{$B$} & \multicolumn{1}{c}{} & \multicolumn{1}{c}{} & \multicolumn{1}{c}{$C$} \\
    \hline
    2   & 4 & 5 & 3 \\
    3   & 5 & 6 & 1 \\
    {\bf\textcolor{blue}{6}}   & 6 & 7 & {\bf\textcolor{blue}{4}} \\
    {\bf\textcolor{blue}{7}}   & 8 & 8 & {\bf\textcolor{blue}{5}} \\
    \hline
\end{tabular} \, ,
\quad T' =
\begin{tabular}{|ccc|c|}
    \multicolumn{1}{c}{$A$} & \multicolumn{1}{c}{} & \multicolumn{1}{c}{} & \multicolumn{1}{c}{} \\
    \hline
    2   & 4 & 5 & 3 \\
    3   & 5 & 6 & 1 \\
    4   & 6 & 7 & 6 \\
    5   & {\bf\textcolor{blue}{7}} & 8 & {\bf\textcolor{blue}{8}} \\
    \hline
\end{tabular}\, ,
\quad \tilde{T} =
\begin{tabular}{|ccc|c|}
    \multicolumn{1}{c}{$B'$} & \multicolumn{1}{c}{} & \multicolumn{1}{c}{} & \multicolumn{1}{c}{$C'$} \\
    \hline
    2   & 4 & 5 & 3 \\
    3   & 5 & 6 & 1 \\
    4   & 6 & 7 & 6 \\
    5   & 8 & 8 & 7 \\
    \hline
\end{tabular}\ .
\]
Following the notation of the proof of Theorem~\ref{prop:quad}, we label the columns with $A, B$ and $C$. We obtain the tableau $\tilde{T}$ from $T$ by swapping the entries $6,7$ in column $B$ with $4,5$ in column $C$. As a result the leftmost column of $\tilde{T}$ and $T'$ are identical.
\end{example}

We now turn our attention to one of the main results of our paper.

\begin{theorem}\label{thm:SAGBI}
The Pl\"ucker variables $P_I$ form a SAGBI basis for the Pl\"ucker algebra with respect to the weight vectors arising from block diagonal matching fields.
\end{theorem}
Before giving the proof of Theorem~\ref{thm:SAGBI} we recall from \S\ref{sec:core} that
$A_\ell=\mathbb{K}[\inwb(P_I)]_{I \in \mathbf{I}_{k,n}}$. We show that the subspace $[A_{\ell}]_2$ of $A_{\ell}$ spanned by the monomials of degree two has a basis 
which is in bijection with the set of semi-standard tableaux with two columns. 
We fix our notation so that $s$ denotes the size of a minimal generating set for the ideal $J_D = {\rm in}_{w_D}(G_{k,n})$ where $w_D$ is the weight induced by the diagonal matching field.

\begin{proof}[Proof of Theorem~\ref{thm:SAGBI}]
To prove the result we show that $J_{\BLambda} = {\rm in}_{w_\ell}(G_{k,n})$.

Let us begin by showing that $\dim([A_{\ell}]_2) = \dim([A_{0}]_2)$. By a classical result, $[A_{0}]_2$ has a basis given by the set of semi-standard tableaux with two columns and $k$ rows.
By Lemma~\ref{lem:basis_span} and \ref{lem:basis_indep} we have that $[A_{\ell}]_2$ has a basis given by $\Tc_{\ell}$ from Definition~\ref{def:basis}. And by Lemma~\ref{lem:basis_bijection} we have that $T \mapsto S(T)$ is a bijection, as desired.

Note that by Lemmas~\ref{lem:g_kn_gen_set_size} and \ref{lem:J_bl_gen_set_size} we have that $J_{\BLambda}$ and $G_{k,n}$ both have minimal generating sets of size $s$.

First we recall that by \cite[Lemma~11.3]{sturmfels1996grobner} we have that ${\rm in}_{w_\ell}(G_{k,n}) \subseteq J_{\BLambda}$. 
Let $f_1, \dots, f_s$ be a collection of homogenous quadratic polynomials which minimally generate $G_{k,n}$. By Lemma~\ref{lem:min_gen_set_alg} there exists a set of homogenous quadratic polynomials $\{g_1, \dots, g_s\} \subset G_{k,n}$ such that $\{\init_{w_\ell}(g_1), \dots, \init_{w_\ell}(g_s)\}$ is a minimal generating set for the ideal $H = \langle \init_{w_\ell}(g_1), \dots, \init_{w_\ell}(g_s) \rangle $. We have the chain of inclusions $H \subseteq \init_{w_\ell}(G_{k,n}) \subseteq J_{\BLambda}$. Since $H$ and $J_{\BLambda}$ are both generated by homogeneous polynomials of degree two, let us consider 
\begin{multline*}
    (H)_2 = \{f \in H : f \textrm{ is homogenous of degree two}\} \\
    \subseteq (J_{\BLambda})_2 = \{f \in J_{\BLambda} : f \textrm{ is homogenous of degree two}\}.
\end{multline*}
We have that $(H)_2$ and $(J_{\BLambda})_2$ are $\mathbb{K}$-vector spaces. The minimal generating sets of these ideals consisting of homogenous polynomials are bases for these vector spaces. It follows that both spaces have dimension $s$, hence they are equal.
Since $H$ and $J_{\BLambda}$ are generated by homogeneous polynomials of degree two, we have shown that $J_{\BLambda} = H \subseteq {\rm in}_{w_\ell}(G_{k,n}) \subseteq J_{\BLambda}$. 
\end{proof}
Here we include the three lemmas used to prove Theorem~\ref{thm:SAGBI}. These lemmas count the size of minimal generating sets of the various quadratically generated ideals with which we are working.
\begin{lemma}\label{lem:g_kn_gen_set_size}
The size of a minimal generating set for $G_{k,n}$ is equal to $s$.
\end{lemma}

\begin{proof}

Recall that $J_D$ denotes the ideal of the diagonal matching field. First we recall the component-wise partial order on subsets given by $\{a_1 < \dots < a_k \} \le \{b_1 < \dots < b_k \}$ if and only if $a_i \le b_i $ for all $1 \le i \le k$. From \cite[Theorem 14.16]{MS05} we recall that $J_D$ is generated by binomials
\[
P_I P_J - P_{I \wedge J} P_{I\vee J}
\]
where $I$ and $J$ are incomparable with respect to the component-wise partial order. Observe that every monomial $P_I P_J$ for which $I$ and $J$ are incomparable appears as a term of a unique generator. So any generator $P_I P_J - P_{I \wedge J} P_{I \vee J}$ cannot be written as a linear combination of the other generators since none of the other generators contain $P_I P_J$. Hence this generating set is minimal. We now proceed to show that $s$ is the size of a minimal generating set for $G_{k,n}$.

By \cite[Theorem 14.6]{MS05} we have that $G_{k,n}$ has a Gr\"obner basis and hence a generating set of size $s$. By \cite[Corollary 14.9]{MS05} we have that that the so-called semi-standard monomials form a basis for the Pl\"ucker algebra. The semi-standard monomials are those of the form $P_{I_1} P_{I_2} \dots P_{I_r}$ where $I_1 \le I_2 \le \dots \le I_r$. So we can take the aforementioned generating set for $G_{k,n}$ to be Pl\"ucker relations of the form
\[
P_I P_J + \sum_{(I', J')} P_{I'} P_{J'}
\]
where $I$ and $J$ are incomparable and for each $(I', J')$ appearing in the sum we have $I' \le J'$. Note that all incomparable pairs $P_I P_J$ appear as the leading term of a unique generator and the generators are of degree two. So by a similar argument as above it follows that these polynomials form a minimal generating set for $G_{k,n}$. So a minimal generating set for $G_{k,n}$ has size equal to the number of incomparable pairs $I$ and $J$ which is equal to $s$.
\end{proof}

\begin{lemma}\label{lem:J_bl_gen_set_size}
The size of a minimal generating set for $J_{\BLambda}$ is $s$. In particular $G_\ell = \{P_I P_J - P_{I'} P_{J'} \in J_{\BLambda} : (I,J) \in \Tc_\ell, \ (I', J') \not\in \Tc_\ell \}$ is a minimal generating set for $J_{\BLambda}$.
\end{lemma}

\begin{proof}
First we show $G_\ell$ generates the ideal. By Theorem~\ref{prop:quad}, $J_{\BLambda}$ is generated by quadratic polynomials so it suffices to check that any quadratic polynomial in $J_{\BLambda}$ is a linear combination of elements of $G_\ell$. Since  $J_{\BLambda}$ is a binomial ideal it suffices to check this for binomials. Consider a binomial $P_{I'} P_{J'} - P_{I''} P_{J''} \in  J_{\BLambda}$. Since $[A_\ell]_2$ has a basis given by $\Tc_\ell$ it follows that at most one of $(I',J')$ and $(I'',J'')$ lies in $\Tc_\ell$. If exactly one lies in $\Tc_\ell$ then the polynomial is an element of $G_\ell$ and we are done. Otherwise assume that neither $(I',J')$ nor $(I'',J'')$ lies in $\Tc_\ell$. Since $\Tc_\ell$ is a basis for $[A_\ell]_2$ it follows that $P_{I'} P_{J'}$ and $P_{I''} P_{J''}$ are equal to a unique monomial $P_I P_J$ with $(I,J) \in \Tc_\ell$. Hence 
\[
P_{I'} P_{J'} - P_{I''} P_{J''} = -(P_{I} P_{J} - P_{I'} P_{J'}) + (P_{I} P_{J} - P_{I''} P_{J''}) \in \langle G_\ell \rangle. 
\]

Next we show that $G_\ell$ is minimal. First observe that every monomial of degree two $P_{I'} P_{J'}$ for which $(I,J) \not \in \Tc_\ell$ is a term of a unique element of $G_\ell$. So any generator $P_I P_J - P_{I'} P_{J'} \in G_\ell$ cannot be written as a linear combination of the other generators because none of the other generators contain $P_{I'} P_{J'}$ as a term. So we have shown $G_\ell$ is minimal.

To deduce that $G_\ell$ has size $s$ we begin by noting that the size of $G_\ell$ is precisely the number of $(I,J)$ which do not lie in $\Tc_\ell$. By the above we have that $\Tc_\ell$ is in bijection with semi-standard tableaux with two columns and $k$ rows. The set of semi-standard tableaux with two columns and $k$ rows is precisely the set of all comparable pairs $(I',J')$. Hence a minimal generating set of $J_{\BLambda}$ has size $s$.
\end{proof}

For the final lemma used to prove Theorem~\ref{thm:SAGBI}, we define the following convenient notation. We say that a collection of polynomials is a \emph{minimal generating set} if it is a minimal generating set for the ideal it generates.
\begin{lemma}\label{lem:min_gen_set_alg}
Let $\{f_1, \dots, f_s\} \subset \mathbb{K}[x_1,\ldots,x_n]$ be a collection of quadratic polynomials which is a
minimal generating set for the ideal $I$ and 
let $w \in \mathbb{Z}_{\geq 0}^n$ be a weight vector. Then there exists a set of quadratic polynomials $\{g_1, \dots, g_s\} \subseteq I$ such that $\{\init_w(g_1), \dots, \init_w(g_s)\}$ is a minimal generating set.
\end{lemma}

\begin{algorithm}[H]\label{alg:min_gen_set}
\SetAlgoLined
\SetNlSty{textit}{}{:}
\SetKwInOut{input}{Input}\SetKwInOut{output}{Output}

\input{A quadratic minimal generating set $\{f_1, \dots, f_s\}$ for $I$.}
\output{A set of quadratic polynomials $\{g_1, \dots, g_s\}\subseteq I$,  such that $\{\init_w(g_1), \dots, \init_w(g_s)\}$ is a minimal generating set.}
\BlankLine

$g_i := f_i$ for each $i \in [s]$\;
$t := \max\{u \in [s] : \{\init_w(f_1), \dots, \init_w(f_u)\}$ is a minimal generating set$\} + 1$\;
\BlankLine

\While{$t \le s$}{
\While{$\init_w(g_t) \in \langle \init_w(g_1), \dots, \init_w(g_{t-1}) \rangle$}{
Write $\init_w(g_t) = \sum_{i = 1}^{t-1} r_i \, \init_w(g_i)$ for some $r_i \in \mathbb{K}$ such that the weight of $g_t$ is equal to the weight of $g_i$ for each $i$ where $r_i \neq 0$\;
$g_t \leftarrow g_t - \sum_{i = 1}^{t-1} r_i \, \init_w(g_i)$\;
}
$t \leftarrow t + 1$\;
}
\caption{A quadratic minimal generating set in $\init_w(I)$}
\end{algorithm}

\begin{proof}
We prove Lemma~\ref{lem:min_gen_set_alg} by showing that Algorithm~\ref{alg:min_gen_set} is correct.
First we note that the set 
\[
\{u \in [n] : \init_w(f_1), \dots, \init_w(f_u) \textrm{  
is a minimal generating set}\}
\] 
is non-empty because $\langle \init_w(f_1) \rangle$ is principal, hence minimally generated by $\init_w(f_1)$. Therefore $t$, in line~2, is well-defined. Let us fix this value for $t$ throughout the rest of the proof. We proceed to show that the output of Algorithm~\ref{alg:min_gen_set} is correct by showing that for each $i \in \{t-1, \dots, s \}$ the set $\{ \init_w(g_1), \dots, \init_w(g_i)\}$ is a minimal generating set. We proceed by induction on $i$.

If $i = t-1$ then, by the definition of $t$, we have that $\init_w(f_1), \dots, \init_w(f_{t-1})$ is a minimal generating set. In line~1 we define $g_i = f_i$ for each $i \in [t-1]$. Note that, for each $i \in [t-1]$, the polynomial $g_i$ is not changed throughout the rest of the algorithm. So $\init_w(g_1), \dots, \init_w(g_{t-1})$ is a minimal generating set.

If $i \ge t$, let us assume, by induction, that $\init_w(g_1), \dots, \init(g_{i-1})$ is a minimal generating set. In order to show that the algorithm is correct, we begin by showing that the while-loop, starting on line~4, terminates. Suppose that $\init_w(g_i) \in \langle \init_w(g_1), \dots, \init_w(g_{i-1})\rangle$. 
Note that all polynomials $g_j$ for $j \in [s]$ are homogeneous quadratic polynomials. It follows that the set of polynomials $p \in \langle \init_w(g_1), \dots, \init_w(g_{i-1}) \rangle$ such that all terms of $p$ have the same weight as $\init_w(g_i)$ are linear combinations of $\init_w(g_j)$ where the weight of $g_j$ is the same as the weight of $g_i$ and $j \in [i-1]$. 
So we are able to find $r_i \in \mathbb{K}$, as in line~5, such that $\init_w(g_i) = \sum_{j = 1}^{t-1} r_j \, \init_w(g_j)$. In line~6 we update $g_i$, note that the weight of $g_i - \sum_{j = 1}^{t-1} r_j \, \init_w(g_j)$ is strictly less than the weight of $g_i$. Since the weight of a polynomial is a non-negative integer, it follows that the while loop from line~4 to 7 is executed finitely many times.

Once the while loop terminates we note that $\init_w(g_i) \not\in \langle \init_w(g_1), \dots, \init_w(g_{i-1})\rangle$. To show that $\{\init_w(g_1), \dots, \init_w(g_i)\}$ is a minimal generating set, it remains to show that 
\[\init_w(g_j) \not\in \langle \{\init_w(g_1), \dots, \init_w(g_i) \} \backslash \init_w(g_j) \rangle
\]
for all $j \in [i-1]$.
So suppose by contradiction that this fails for some $j$. Then we can write $\init_w(g_j) = \sum_{k \in [i] \backslash j} r_k \, \init_w(g_k)$ for some $r_k \in \mathbb{K}$ where $r_k = 0$ if the weight of $g_k$ is different to the weight of $g_j$. If $r_i = 0$ then we have shown that $\init_w(g_1), \dots, \init_w(g_{i-1})$ is not a minimal generating set, a contradiction, and so $r_i \neq 0$. We may rearrange this equation to obtain
\[
\init_w(g_i) = \sum_{k = 1}^{i-1} \frac{r_k}{r_i} \, \init_w(g_k), \textrm{ where we define } r_j \textrm{ to be } 1.
\]
This implies that $\init_w(g_i) \in \langle \init_w(g_1), \dots, \init_w(g_{i-1})\rangle $, a contradiction. And so we have shown that $\init_w(g_1), \dots, \init_w(g_i)$ is a minimal generating set.
\end{proof}

As a corollary of the above statements and \cite[Theorem 11.4]{sturmfels1996grobner} we have that:
\begin{corollary} 
Each block diagonal matching field produces a toric degeneration of $\Gr(k,n)$.
\end{corollary}

\begin{remark}
We remark that this result is a generalization of Corollary~1.5 in \cite{KristinFatemeh} for $\Gr(3,n)$. For general Grassmannians, there are other families of combinatorial objects leading to toric degenerations such as Newton--Okounkov bodies \cite{An13,KM16}, plabic graphs \cite{BFFHL,KM16} and cluster algebra \cite{rietsch2017newton,Fatemeh4}.
All such degenerations can be realized as Gr\"obner degenerations, nevertheless, this is not true in general; See e.g.
\cite{kateri2015family}.
Moreover, our combinatorial description of SAGBI bases of Grassmannians leads to analogous results for flag varieties. Using the combinatorial tools of matching field tableaux we have provided a family of toric degenerations of flag varieties,  Schubert varieties and Richardson varieties in \cite{OllieFatemeh2,OllieFatemeh3,Ollie4}. 
\end{remark}

Following \cite{KristinFatemeh}, we define the matching field polytope as follows. Given a $k\times n$ matching field, the {\em matching field polytope} is the convex hull of the points in $\mathbb{R}^{k\times n}$ associated to the $k$-subsets of $[n]$. See \cite[Section~5]{KristinFatemeh} for more details. For a coherent matching field $\Lambda$, this polytope coincides with the polytope of the toric variety defined by the ideal of the matching field $J_\Lambda$. Hence, its combinatorial invariants carry a lot of information about the variety. For example,  the normalized lattice volume of such polytope is equal to the degree of its corresponding variety. Moreover, by comparing the polytopes of different toric varieties arising from our construction, we can check whether the varieties are non-isomorphic. 

\medskip

Up to isomorphism, there are seven polytopes associated to trop$(\Gr(3,6))$ and four of them can be obtained as polytopes of block diagonal matching fields. See \cite{herrmann2009draw,Akihiro} for further details. 

\smallskip

Here, we summarize our computational results on matching field polytopes.

\begin{remark}
Using {\tt polymake}~\cite{polymake:2017} we computed the f-vectors of the polytopes associated to toric ideals of block diagonal matching fields, see Table~\ref{tab:toric_polytopes}. In particular, Table~\ref{tab:toric_polytopes} shows that {\it almost all} toric ideals obtained by our construction are non-isomorphic. Note that for some values of $\ell$, the f-vector is the same but the polytopes are non-isomorphic. For instance, in the cases of $\Gr(3,6)$ and $\Gr(4,7)$ the polytopes associated to the block diagonal matching fields $\BLambda$, for $\ell = 1$ and $\ell = 2$ have the same f-vectors, however, by computing their face lattices, we can show they are non-isomorphic. For some values of $\ell$, the matching field ideal is trivially isomorphic to the diagonal case, so we list only the f-vectors for the cases $0\leq\ell\leq n-k+1$.  In  \cite{Akihiro} we
study these polytopes from a geometric point of view.
\end{remark}

\section{Toric degeneration of Schubert varieties inside Grassmannians}\label{sec:Schubert}
In this section, we apply our results from \S\ref{sec:Grassmannian} for Grassmannains to provide a family of toric degenerations for Schubert varieties. Our aim is to answer the following question which is a reformulation of the {\em Degeneration Problem} posed by Caldero in \cite{caldero2002toric}, in our setting. 

\begin{question}\label{question:grassmannian}
Characterize non-zero toric ideals of type $G_{k,n,\ell,w}$. 
\end{question}
We provide a complete answer to Question~\ref{question:grassmannian}. In particular, we give a complete characterization of toric ideals of type $G_{k,n,\ell,w}$ from Definition~\ref{def:ideals}. We will first distinguish such ideals which are non-zero in Proposition~\ref{lem:Zero_Gr} and then in Theorem~\ref{Intro:Grassmannian} we provide a list of combinatorial conditions which lead to toric ideals.

\begin{proposition}\label{lem:Zero_Gr}
The ideal $G_{k,n,\ell,w}$ is zero if and only if $w \in Z_{k,n}$, where
\[
Z_{k,n}=\{(1,2,\ldots,k-1,i):\ k\leq i\leq n\} \cup
\{(1,\ldots,\hat{i},\ldots,k,k+1):\ 1\leq i\leq k-1\}.
\]
\end{proposition}

\begin{proof}
To begin we show that $G_{k,n,\ell,w}$ is zero for each $w \in Z_{k,n}$. We distinguish two cases:
\medskip

\textbf{Case 1.} Let $w = (1, 2, \dots, k-1, i)$ for some $k \le i \le n$. Then the only variables $P_I$ which do not vanish in $G_{k,n,\ell,w}$ are indexed by:
\[I = \{1, 2, \dots, k-1, j \} \text{ for some } k \le j \le i.\]
Suppose by contradiction that $G_{k,n,\ell,w}$ is non-zero. Then there is a non-trivial relation $P_I P_J = P_{I'} P_{J'}$ in $\inwb(G_{k,n})$ for which $P_I$ and $P_J$ do not vanish. Write $I = \{1, \dots, k-1, j_1\}$ and $J = \{1, \dots, k-1, j_2\}$. Note that as multisets $I \cup J$ and $I' \cup J'$ are identical and so $I'$ and $J'$ each contain $1, \dots, k-1$. So either $I' = I$ or $I' = J$, hence the relation is trivial, a contradiction. Therefore, $G_{k,n,\ell,w}$ is zero.

\medskip

\textbf{Case 2.} Let $w = (1, 2, \dots, \hat{i}, \dots, k, k+1)$ for some $1 \le i \le k-1$. Then the only variables $P_I$ which do not vanish in $G_{k,n,\ell,w}$ are indexed by:
\[I = \{ 1, 2, \dots, \hat{j}, \dots, k+1\} \text{ for some } i \le j \le k+1.\]
Suppose by contradiction that $G_{k,n,\ell,w}$ is non-zero. Then there is a non-trivial relation $P_I P_J = P_{I'} P_{J'}$ in $\inwb(G_{k,n})$ for which $P_I$ and $P_J$ do not vanish. Write $I = \{1, \dots, \hat{j}_1, \dots, k+1\}$ and $J = \{1, \dots, \hat{j}_2, \dots, k+1\}$. Note that as multisets $I \cup J$ and $I' \cup J'$ are identical and so $I'$ and $J'$ each contain $1, \dots, \hat{j}_1, \dots, \hat{j}_2, \dots, k+1$. So either $I' = I$ or $I' = J$, hence the relation is trivial, a contradiction. So $G_{k,n,\ell,w}$ is zero.

\medskip

Conversely, let $w = (w_1, \dots, w_k)$ be a Grassmannian permutation not in $Z_{k,n}$. We will show that $G_{k,n,\ell,w}$ is non-zero. Note that for $k = 1$ the result is trivial. The cases $k = 2$ and $k=3$ hold by Lemma~\ref{lem:zero_k=2} and Lemma~\ref{lem:zero_k=3}, respectively. So it remains to show the result for $k \ge 4$.

\smallskip

Let $k \ge 4$. We will find a relation $P_I P_J = P_{I'} P_{J'}$ in $\inwb(G_{k,n})$ for which $P_I$ and $P_J$ do not vanish in $G_{k,n,\ell,w}$. First, note that there exists $1 \le i \le k-1$ such that $w_i \neq i$, otherwise $w = (1, 2, \dots, k-1, j) \in Z_{k,n}$ for some $k \le j \le n$. Secondly, there exists $i < j \le k+1$ such that $\{1, 2, \dots, \hat{i}, \dots, \hat{j}, \dots, k+2 \} \le w$, otherwise $w = (1, 2, \dots, \hat{i}, \dots, k+1) \in Z_{k,n}$. Let
\[I = \{1, \dots, k-2, k-1, k+1 \} \ \text{and}\ J = \{1, \dots, k-2, k, k+2 \}. \]
Since $j \le k+1$ and $i \le k-1$ we have $I,J \le \{1, 2, \dots, \hat{i}, \dots, \hat{j}, \dots, k+2 \}$. Therefore $P_I$ and $P_J$ do not vanish in $G_{k,n,\ell,w}$. Let
\[ 
I' = \{1, \dots, k-2, k-1, k+2 \}\ \text{and}\ J' = \{1, \dots, k-2, k, k+1 \}.
\]
Now we show that $P_I P_J = P_{I'} P_{J'}$ is a relation in $\inwb(G_{k,n})$. We recall that $k \ge 4$. Since each set $I, I', J, J'$ contains $1$ and $2$, it follows that $\BLambda(I) = \BLambda(I') = \BLambda(J) = \BLambda(J')$. And so, in the ordering of the sets $I,I', J, J'$ with respect to the matching field, the final two entries of each ordered set are the largest two elements which are in increasing order. Since $I, I', J, J'$ agree on all entries except the final two, it follows that $P_I P_J = P_{I'} P_{J'}$ is a relation in $\inwb(G_{k,n})$. \qed

\end{proof}

\begin{lemma}\label{lem:zero_k=2}
Fix $n$ and let $k = 2$. Let $\B = \BLambda$ be a block diagonal matching field. If $w = (w_1, w_2) \not \in Z_{2,n}$ is a Grassmannian permutation then $G_{k,n,\ell,w}$ is non-zero.
\end{lemma}

\begin{proof}
We write $\B'$ for the restriction of $\B$ to $\{1,2,3,4\}$. Explicitly, $\B' = \B_{\ell'}$ is the block diagonal matching field on $\{1,2,3,4\}$ with $\ell' = \min\{4, \ell\}$. Let $v = (2,4)$.
Note that $Z_{2,n} = \{(1,2), (1,3), \dots, (1,n), (2,3) \}$ so if $w \not \in Z_{2,n}$ then we have $w_1 \ge 2$ and $w_2 \ge 4$. Therefore $v \le w$. For each possible $\ell' \in \{0,1,2,3,4 \}$ we see from Table~\ref{fig:Gr24} that $G_{2,4,\ell',v}$ is non-zero which completes the proof. \qed

\end{proof}
\begin{table}[]
    \begin{center}
   \resizebox{6cm}{!}{ \begin{tabular}{|c|l|}\hline
              Matching fields & Toric permutations\\\hline 
        Diagonal        & 24 34 \\
        (123$\mid$4)    & 34 \\
        (12$\mid$34)    & 24 34 \\
        (1$\mid$234)    & 34 \\ \hline\hline
        Zero            & 12 13 14 23 \\ \hline
   \end{tabular}
   }     \end{center}
    
    \caption{For each block diagonal matching field $\BLambda$ for $\Gr(2,4)$, we list the Grassmannian permutations $w$ for which $G_{2,4,\ell,w}$ is either toric or zero. }\label{fig:Gr24}
\end{table}

\begin{table}[]
    \begin{center}
   \resizebox{13cm}{!}{ \begin{tabular}{|c|l|}
    \hline
              Matching fields & Toric permutations\\\hline
        Diagonal        & 135  235  145  245  345  136  236  146  246  346  156  256  356  456  \\
        $\B_5$  & 135  235  145  245  345  136  236  146  246  346  156  256  356  456  \\
        $\B_4$  & 135  235  145  245  345  136  236  146  246  346  156  456            \\
        $\B_3$  & 135  235  145  345  136  236  146  346  156  356  456                 \\
        $\B_2$  & 135  235  145  245  345  136  236  146  246  346  156  256  456       \\
        $\B_1$  & 135  235  145  345  136  236  146  346  156  456                      \\ \hline\hline
        Zero            & 123  124  125  126  134  234                                          \\ \hline
        \end{tabular}
    }\end{center}
    
    \caption{For each block diagonal matching field $\BLambda$ for $\Gr(3,6)$, we list the Grassmannian permutations $w$ for which $G_{3,6,\ell,w}$ is toric. The last row of the table represents all permutations leading to the zero ideal.}\label{fig:Gr36}
\end{table}
\begin{lemma}\label{lem:zero_k=3}
Fix $n$ and let $k = 3$. Let $\B = \BLambda$ be a block diagonal matching field. If $w = (w_1, w_2, w_3) \not \in Z_{3,n}$ is a Grassmannian permutation then $G_{k,n,\ell,w}$ is non-zero.
\end{lemma}

\begin{proof}
Denote by $\B'$ the block diagonal matching field $\B$ restricted to $\{1, \dots, 6 \}$. If each entry $w_i \in \{1, \dots, 6 \}$, then $w$ is a Grassmannian permutation for $\Gr(3,6)$ which is not in $Z_{3,6}$. In Table~\ref{fig:Gr36} we have verified the result for $\Gr(3,6)$, and so $G_{3,6,\ell',w}$ is non-zero. So there is a relation $P_I P_J = P_{I'} P_{J'}$ 
such that $P_I$ and $P_J$ do not vanish in $G_{3,6,\ell',w}$. 
Similarly $P_I$ and $P_J$ do not vanish in $G_{3,n,\ell,w}$, hence $G_{3,n,\ell,w}\neq 0$.

On the other hand, if we do not have each $w_i $ in $\{1, \dots, 6 \}$, then we consider $v = (1, 3, 6) \not\in Z_{3,n}$. Since $w \not\in Z_{3,n}$ we must have $w_1 \ge 1, w_2 \ge 3$ and by assumption $w_3 > 6$. Therefore $v \le w$. By the calculation in Table~\ref{fig:Gr36} we have that $G_{3,6,\ell',v}$ is non-zero and so there is a relation $P_I P_J = P_{I'} P_{J'}$ 
such that $P_I$ and $P_J$ do not vanish in $G_{3,6,\ell',v}$. Since $I,J \le v \le w$ we have that $P_I$ and $P_J$ do not vanish in $G_{3,n,\ell,w}$. Therefore, $G_{3,n,\ell,w}$ is non-zero.\end{proof}

\begin{proposition}\label{thm:Toric_Diagonal}
For the diagonal matching field, all non-zero ideals $G_{k,n,\ell,w}$ are toric. 
\end{proposition}
\begin{proof} For the diagonal matching field we have $\ell=0$. Suppose $G_{k,n,\ell,w}$ is non-zero and write $w = (w_1 < \dots < w_k)$. Let $P_I$ and $P_J$ be variables which do not vanish in $G_{k,n,\ell,w}$ and suppose that $I$ and $J$ are incomparable. Write $I = \{i_1 < \dots < i_k \}$ and $J = \{j_1 < \dots < j_k \}$. It suffices to show that $P_{I \wedge J}$ and $P_{I \vee J}$ do not vanish in $G_{k,n,\ell,w}$, where $I \wedge J = \{ \min(i_1, j_1) < \dots < \min(i_k, j_k) \}$ and $I \vee J = \{\max(i_1,j_1) < \dots < \max(i_k, j_k) \}$. However, this is immediate since $P_I$ and $P_J$ do not vanish. Hence, for every $t \in [k]$ we have $i_t \le w_t$ and $j_t \le w_t$. Therefore $P_I P_J - P_{I \wedge J} P_{I \vee J}$ is a relation in $G_{k,n,\ell,w}$ and in particular the ideal contains no monomials, hence it is toric.
\end{proof}

\begin{remark}
The ideal $G_{k,n,\ell,w}$ is the Hibi ideal \cite{hibi1987distributive} whose generators $P_I P_J - P_{I \wedge J} P_{I \vee J}$ can be read from the distributive lattice $\{I \subseteq [n] : P_I \text{ does not vanish in } G_{k,n,\ell,w}\}$.
\end{remark}

\begin{theorem}\label{Intro:Grassmannian}
Let $w=(w_1,\ldots,w_k)$ be a permutation of $k$ indices in $[n]$.  
Then a non-zero ideal 
$G_{k,n,\ell,w}$ is non-toric if and only if the following hold:
\begin{itemize}
\item[$(i)$]  $\ell \neq 0$
\item[$(ii)$] $w_1\in\{2,\ldots,n-k\}\backslash\{\ell\}$
\item[$(iii)$] $\{w_2,\ldots,w_k\}\subseteq \{\ell+1,\ldots,n\}\backslash\{w_1+1\}$
\end{itemize}
\end{theorem}

\begin{proof}
Fix $\ell \in \{0,1, \dots, n-1 \}$. If $\ell = 0$ then by Proposition~\ref{thm:Toric_Diagonal} we have that for any $w$, $G_{k,n,\ell,w}$ is either toric or zero. So we may assume that $\ell > 0$.

Suppose that $w$ satisfies the given conditions and write $\B$ for $\BLambda$.  We will construct a relation $P_I P_J = P_{I'} P_{J'}$ for which only $P_{I'}$ vanishes in $G_{k,n,\ell,w}$. For ease of notation let $u = (w_3, \dots, w_k)$. Suppose that $w_1 \in B_{\ell,1}$. Then consider the following relation among the Pl\"ucker variables
$$
P_{w_2, w_1, u} \, P_{1, w_1+1, u} = 
P_{w_2, w_1+1, u} \, P_{1, w_1, u}.
$$
Note that the relation is given with indices ordered according to the matching field $\B$.
By assumption, $1 < w_1 < w_1+1 < w_2 < \dots < w_k$ so in the above expression the only variable to vanish in $G_{k,n,\ell,w}$ is $P_{w_2, w_1+1, u}$ because
$$(1, w_1, u) < (1, w_1+1, u) < (w_1, w_2, u) = w$$
whereas $(w_1+1, w_2, u) \not \leq (w_1, w_2, u)$. Therefore, $G_{k,n,\ell,w}$ is non-toric as it contains the monomial 
$P_{w_2, w_1, u} P_{1, w_1+1, u}$.
Now suppose that $w_1 \in B_{\ell,2}$. Then we consider the relation 
$$
P_{w_1+1, 1, u} \, P_{w_1, w_2, u} = 
P_{w_1+1, w_2, u} \, P_{w_1, 1, u}
$$
among the Pl\"ucker variables. Similarly, in the above relation the only variable to vanish in $G_{k,n,\ell,w}$ is $P_{w_1+1, w_2, u}$. So $G_{k,n,\ell,w}$ contains the monomial $P_{w_1+1, 1, u} \, P_{w_1, w_2, u}$ and hence is non-toric.

For the converse, suppose $G_{k,n,\ell,w}$ is non-toric. We proceed by induction on $k$ and $n$, reducing to the cases with $k = 3$ and $n = 6$ in which $w$ has the desired form by direct computation.

Suppose $k > 3$. We reduce to the case $k = 3$ as follows. Write $w = (w_1 < \dots < w_k)$. By assumption $G_{k,n,\ell,w}$ is non-toric and so contains a monomial $P_I P_J$ for some subsets $I = \{i_1, \dots, i_k \}$ and $J = \{j_1, \dots, j_k \}$. In $\inwb(G_{k,n})$ this monomial belongs to a relation $P_I P_J = P_{I'} P_{J'}$ where at least one of the variables $P_{I'}, P_{J'}$ vanishes in $G_{k,n,\ell,w}$. Assume that $I' = \{i_1', \dots, i_k' \}$ and $J' = \{j_1', \dots, j_k' \}$. Now we take cases on $i_k' \in \{i_k, j_k \}$.

\smallskip

\textbf{Case 1.} Let $i_k' = i_k$. Hence $j_k' = j_k$. By assumption neither $P_I$ nor $P_J$ vanishes in $G_{k,n,\ell,w}$ so $i_k, j_k \le w_k$. Therefore $P_{I\setminus i_k} P_{J \setminus j_k}$ does not vanish in $G_{k-1, n,\ell,w'}$ where $w' = (w_1, \dots, w_{k-1})$. However in $G_{k,n,\ell,w}$ we must have that either $P_{I'}$ or $P_{J'}$ vanishes so either $I' \not \leq w$ or $J' \not \leq w$. Since $i_k', j_k' \leq w_k$ then either $I'\setminus i_k' \not \leq (w_1, \dots, w_{k-1})$ or $J'\setminus j_k' \not \leq (w_1, \dots, w_{k-1})$. Hence $P_{I'} P_{J'}$ vanishes in $G_{k-1,n,\ell,w'}$. And so the relation $P_{I\setminus i_k} P_{J\setminus j_k} = P_{I'\setminus i_k'} P_{J'\setminus j_k'}$ gives rise to the monomial $P_{I\setminus i_k} P_{J\setminus j_k}$ in $G_{k-1,n,\ell,w'}$.

\smallskip

\textbf{Case 2.} Let $i_k' = j_k$ and $i_k' \neq i_k$. Now if we have $I\setminus i_k \neq I'\setminus i_k'$ then we may use the same argument above to show that $P_{I\setminus i_k} P_{J\setminus j_k}$ is a monomial in $G_{k-1, n,\ell,w'}$. Otherwise $I\setminus i_k = I'\setminus i_k'$ and hence $J\setminus j_k = J'\setminus j_k'$. But then we have $I' \le w$ and $J' \le w$ since $i_k, i_k', j_k, j_k' \le w_k$. And so $P_{I'} P_{J'}$ does not vanish in $G_{k,n,\ell,w}$, a contradiction. So we have shown that all permutations $w$ for which $G_{k,n,\ell,w}$ are non-toric arise from those in {$G_{k-1,n,\ell,w'}$} by deleting the last entry.

So we may assume that $k = 3$. Now, we consider the monomial $P_I P_J$ contained in $G_{k,n,\ell,w}$. Since we have that $\lvert I \cup J \rvert \le 6$, we may reduce to the corresponding case with $k=3$ and $n = 6$. See Table~\ref{fig:Gr36} for the list of toric ideals among $G_{3,6,\ell,w}$.
\end{proof}

\begin{example}
Here, we provide illustrative examples of binomial relations used throughout the proof of Theorem~\ref{Intro:Grassmannian}. 
Let $n = 9$, $k = 4$, $\B_{4}$ and $w = (2,5,8,9)$. Note that by the theorem we have that $G_{k,n,\ell,w}$ is non-toric. To prove this, the proof constructs the relation
$
P_{5289} P_{1389} = P_{5389} P_{1289}
$
in $G_{k,n,\ell,w}$.
Note that the left hand side does not vanish in $G_{k,n,\ell,w}$ however $(5,3,8,9) \not\leq (2,5,8,9)$ so the right hand side vanishes.

Let $n,k,\ell$ and $w$ be as above. We now consider the converse part of the proof of Theorem~\ref{Intro:Grassmannian}. Suppose we are told that $G_{k,n,\ell,w}$ is non-toric and in particular we are given that the ideal contains the monomial $P_{5278} P_{1389}$. Suppose it is obtained from the relation
$
P_{5278} P_{1389} = P_{5389} P_{1278}
$
where the right hand side vanishes but the left hand side does not. We now consider $G_{k-1, n, \ell, w'}$ where $w' = (2,5,8)$. This is non-toric because it contains the monomial $P_{527} P_{128}$ arising from 
$
P_{527} P_{138} = P_{538} P_{127}.
$
To reduce to the case with $n=6$ we consider the entries in the indices of the above relation. These entries are $E = \{1,2,3,5,7,8\}$. So, under the order-preserving bijection between $E$ and $\{1,2,3,4,5,6\}$, the above is equivalent to looking at $G_{3,6,\ell',w''}$ where $w'' = (2,4,6)$ and $\ell' = 3$. Under the bijection, the relation becomes
$
P_{425} P_{136} = P_{436} P_{125}.
$
Note that the right hand side vanishes but the left hand side does not. So $w''$ satisfies the criteria of Theorem~\ref{Intro:Grassmannian} and so, under the bijection, $w'$ and $w$ also satisfy these criteria.
\end{example}


\begin{example}
In Tables~\ref{fig:Gr24} and \ref{fig:Gr36} we list all Grassmannian permutations $w$, and all block diagonal matching fields $\BLambda$ whose corresponding ideal $G_{2,4,\ell,w}$ and $G_{3,6,\ell,w}$ is either toric or zero. Also, we can explicitly calculate the number and so the percentage of pairs $(\ell, w)$ for which $G_{k,n,\ell,w}$ is toric,  
see Tables \ref{fig:toric_zero_non-toric} and 
\ref{fig:percent_toric}. The code used to perform these calculations is available on Github:
\begin{center}
	\href{https://github.com/ollieclarke8787/toric_degenerations_gr}{https://github.com/ollieclarke8787/toric\_degenerations\_gr}
\end{center}
This repository also contains instructions for running and generating new code to perform calculations for each Grasmmannian.
\begin{table}
    \centering
     \resizebox{10cm}{!}{\begin{tabular}{|c|c|c|c|}
        \hline
        Toric    & \multicolumn{3}{|c|}{$n$} \\
        \hline
        $k$ & 4     & 5     & 6         \\
        \hline
        2   & 6     & 17    & 34  \\
        3   &       & 23    & 74  \\
        4   &       &       & 52   \\
        \hline
    \end{tabular}\quad
    \begin{tabular}{|c|c|c|c|}
        \hline
        Zero    & \multicolumn{3}{|c|}{$n$} \\
        \hline
        $k$ & 4     & 5     & 6     \\
        \hline
        2   & 16    & 25    & 36  \\
        3   &       & 25    & 36  \\
        4   &       &       & 36   \\
        \hline
    \end{tabular}\quad
    \begin{tabular}{|c|c|c|c|}
        \hline
        Non-Toric    & \multicolumn{3}{|c|}{$n$} \\
        \hline
        $k$ & 4     & 5     & 6         \\
        \hline
        2   & 2     & 8     & 20  \\
        3   &       & 2     & 10  \\
        4   &       &       & 2   \\
        \hline
    \end{tabular}
    }
    \caption{For small $k, n$ we list the number of toric, zero and non-toric ideals of form $G_{k,n,\ell,w}$.}
    \label{fig:toric_zero_non-toric}
\end{table}\end{example}
\begin{remark}
Fix $k$ and $n$. We can explicitly count the number of distinct pairs $(\ell, w)$ for which $G_{k,n,\ell,w}$ is zero. By Proposition~\ref{lem:Zero_Gr} there are exactly $n^2$ such pairs. Similarly, we can count the number of pairs $(\ell, w)$ for which $G_{k,n,\ell,w}$ is non-toric. By Theorem~\ref{Intro:Grassmannian} there are exactly $2 \binom{n-1}{k+1}$ such pairs. In total there are $n \binom{n}{k}$ distinct pairs $(\ell, w)$, and so there are $n \binom{n}{k} - 2\binom{n-1}{k+1} - n^2$ pairs which give rise to toric ideals of type $G_{k,n,\ell,w}$ inside Schubert varieties. 
\end{remark}
\begin{table}
    \centering
    \resizebox{13cm}{!}{\begin{tabular}{|c|ccccccccccccccccccc|}
        \hline
        Toric & \multicolumn{19}{c|}{$n$}                                       \\ \hline
        $k$   &4    &5  & 6  & 7  & 8  & 9  & 10 & 11 & 12 & 13 & 14 & 15 & 16 & 17 & 18 & 19 & 20 & 21 & 22 \\ \hline
        2     &25   &34 & 38 & 39 & 40 & 40 & 40 & 40 & 40 & 40 & 40 & 39 & 39 & 39 & 39 & 39 & 38 & 38 & 38 \\
        3     &     &46 & 62 & 68 & 70 & 71 & 71 & 70 & 70 & 69 & 68 & 67 & 67 & 66 & 65 & 65 & 64 & 64 & 63 \\
        4     &     &   & 58 & 75 & 81 & 83 & 83 & 83 & 82 & 81 & 80 & 79 & 78 & 78 & 77 & 76 & 76 & 75 & 74 \\
        5     &     &   &    & 65 & 83 & 88 & 89 & 89 & 89 & 88 & 87 & 86 & 85 & 85 & 84 & 83 & 82 & 82 & 81 \\
        6     &     &   &    &    & 71 & 87 & 92 & 93 & 93 & 92 & 91 & 91 & 90 & 89 & 88 & 88 & 87 & 86 & 86 \\
        7     &     &   &    &    &    & 74 & 90 & 94 & 95 & 95 & 94 & 94 & 93 & 92 & 91 & 91 & 90 & 90 & 89 \\
        8     &     &   &    &    &    &    & 77 & 92 & 96 & 96 & 96 & 96 & 95 & 94 & 94 & 93 & 93 & 92 & 92 \\
        9     &     &   &    &    &    &    &    & 80 & 94 & 97 & 97 & 97 & 97 & 96 & 96 & 95 & 94 & 94 & 94 \\
        10    &     &   &    &    &    &    &    &    & 82 & 95 & 97 & 98 & 98 & 97 & 97 & 96 & 96 & 95 & 95 \\
        11    &     &   &    &    &    &    &    &    &    & 83 & 96 & 98 & 98 & 98 & 98 & 97 & 97 & 97 & 96 \\
        12    &     &   &    &    &    &    &    &    &    &    & 84 & 96 & 98 & 99 & 98 & 98 & 98 & 97 & 97 \\
        13    &     &   &    &    &    &    &    &    &    &    &    & 86 & 97 & 99 & 99 & 99 & 98 & 98 & 98 \\
        14    &     &   &    &    &    &    &    &    &    &    &    &    & 87 & 97 & 99 & 99 & 99 & 99 & 98 \\
        15    &     &   &    &    &    &    &    &    &    &    &    &    &    & 87 & 98 & 99 & 99 & 99 & 99 \\
        16    &     &   &    &    &    &    &    &    &    &    &    &    &    &    & 88 & 98 & 99 & 99 & 99 \\
        17    &     &   &    &    &    &    &    &    &    &    &    &    &    &    &    & 89 & 98 & 99 & 99 \\
        18    &     &   &    &    &    &    &    &    &    &    &    &    &    &    &    &    & 89 & 98 & 99 \\
        19    &     &   &    &    &    &    &    &    &    &    &    &    &    &    &    &    &    & 90 & 98 \\
        20    &     &   &    &    &    &    &    &    &    &    &    &    &    &    &    &    &    &    & 90 \\ \hline
    \end{tabular}
    }
    \caption{For each $k, n$ we calculate the percentage of pairs $(\ell, w)$ for which $G_{k,n,\ell,w}$ is toric. }
    \label{fig:percent_toric}
\end{table}

\bigskip

\bibliographystyle{alpha} 
\bibliography{Trop1}

\bigskip
\bigskip

\noindent
\footnotesize {\bf Authors' addresses:}

\medskip

\noindent University of Bristol, School of Mathematics,
BS8 1TW, Bristol, UK
\\
\noindent  E-mail addresses: {\tt oliver.clarke@bristol.ac.uk}

\medskip

\noindent
Department of Mathematics: Algebra and Geometry, Ghent University, 9000 Gent, Belgium \\
Department of Mathematics and Statistics, 
UiT – The Arctic University of Norway, 9037 Troms\o, Norway
\\ E-mail address: {\tt fatemeh.mohammadi@ugent.be}

\bigskip

\begin{table}
    \centering
    \footnotesize
        \begin{tabular}{|c|c|c|c|}
        \hline
        $(k,n)$ & ${\ell}$ & The $f$-vector of the toric polytope \\
        \hline
        \hline
        \multirow{3}{*}{(3,6)} & 0,3 & 20 122 372 670 766 571 276 83 14   \\
        \cline{2-3}
         & 1 & 20 122 376 690 807 615 302 91 15   \\
        \cline{2-3}
         & 2* & 20 122 376 690 807 615 302 91 15   \\
         \cline{2-3}
         & 4 & 20 122 378 701 832 645 322 98 16   \\
         \hline
        \hline
        \multirow{5}{*}{(3,7)} & 0 & 35 329 1514 4177 7599 9579 8573 5485 2487 778 159 19   \\
        \cline{2-3}
         & 1 & 35 329 1546 4411 8352 10977 10221 6762 3136 986 197 22   \\
        \cline{2-3}
         & 2 & 35 329 1548 4424 8388 11032 10271 6789 3144 987 197 22   \\
         \cline{2-3}
         & 3 & 35 329 1528 4276 7907 10132 9204 5959 2721 851 172 20   \\
         \cline{2-3}
         & 4 & 35 329 1535 4329 8084 10474 9625 6301 2904 913 184 21   \\
         \cline{2-3}
         & 5 & 35 329 1555 4483 8606 11495 10893 7336 3458 1100 220 24   \\
         \hline
        \multirow{3}{*}{(4,7)} & 0,3 & 35 329 1514 4177 7599 9579 8573 5485 2487 778 159 19 \\
        \cline{2-3}
         & 1 & 35 329 1528 4276 7907 10132 9204 5959 2721 851 172 20 \\
        \cline{2-3}
         & 2* & 35 329 1528 4276 7907 10132 9204 5959 2721 851 172 20 \\
         \cline{2-3}
         & 4 & 35 329 1535 4329 8084 10474 9625 6301 2904 913 184 21 \\
        \hline
        \hline
                \multirow{7}{*}{(3,8)} & 0 & \begin{tabular}[c]{@{}c@{}}56 756 4852 18664 48026 87804 118120 119262 \\ 91204 52844 23004 7384 1690 260 24\end{tabular} \\
        \cline{2-3}
        & 1 & \begin{tabular}[c]{@{}c@{}}56 756 4997 20102 54503 105309 149704 159294 \\ 127702 76922 34389 11138 2507 365 30\end{tabular} \\
        \cline{2-3}
        & 2 & \begin{tabular}[c]{@{}c@{}}56 756 5014 20251 55087 106656 151727 161359 \\ 129151 77611 34601 11176 2510 365 30\end{tabular} \\
        \cline{2-3}
        & 3 & \begin{tabular}[c]{@{}c@{}}56 756 4955 19647 52288 98898 137416 142869 \\ 112015 66149 29118 9350 2109 313 27\end{tabular} \\
        \cline{2-3}
        & 4 & \begin{tabular}[c]{@{}c@{}}56 756 4932 19458 51605 97483 135591 141406 \\ 111378 66126 29258 9430 2129 315 27\end{tabular} \\
        \cline{2-3}
        & 5 & \begin{tabular}[c]{@{}c@{}}56 756 4994 20020 53925 103256 145248 152866 \\ 121291 72438 32201 10412 2352 346 29\end{tabular} \\
        \cline{2-3}
        & 6 & \begin{tabular}[c]{@{}c@{}}56 756 5014 20302 55557 108606 156529 169146 \\ 137866 84464 38376 12598 2857 414 33\end{tabular} \\
        \hline
        \multirow{6}{*}{(4,8)} & 0 & \begin{tabular}[c]{@{}c@{}}70 1097 7901 33641 95366 192200 286329 322731 278740 \\ 185312 94561 36643 10565 2189 307 26\end{tabular} \\
        \cline{2-3}
        & 1 & \begin{tabular}[c]{@{}c@{}}70 1097 8050 35219 102954 214160 329032 381669 338135 \\ 229544 118884 46386 13324 2711 366 29\end{tabular} \\
        \cline{2-3}
        & 2 & \begin{tabular}[c]{@{}c@{}}70 1097 8057 35293 103297 215086 330656 383616 339765 \\ 230497 119265 46485 13339 2712 366 29\end{tabular} \\
        \cline{2-3}
        & 3 & \begin{tabular}[c]{@{}c@{}}70 1097 7950 34187 98077 200195 302020 344432 300536 \\ 201423 103327 40109 11532 2369 327 27\end{tabular} \\
        \cline{2-3}
        & 4 & \begin{tabular}[c]{@{}c@{}}70 1097 8021 34837 100844 207461 315162 361710 317488 \\ 213949 110269 42946 12359 2531 346 28\end{tabular} \\
        \cline{2-3}
        & 5 & \begin{tabular}[c]{@{}c@{}}70 1097 8070 35509 104698 220164 342491 402630 361607 \\ 248745 130380 51365 14840 3018 403 31\end{tabular} \\
        \hline
    \end{tabular}
    
    \caption{For each $\Gr(k,n)$ and matching field $\BLambda$ we calculate the f-vector of the toric polytope associated to the matching field $\BLambda$. The f-vector for the rows (*) are the same as the row above, however these polytopes have non-isomorphic face lattices and so define non-isomorphic toric varieties.
}
    \label{tab:toric_polytopes}
\end{table}

\end{document}